%% file: Walton_Whittaker.tex
\documentclass[12pt,a4paper]{amsart}
\usepackage{amssymb}
\usepackage{amsthm}
\usepackage{mathtools} 
\usepackage[total={6.6in,9.2in}, top=1.5in, left=0.85in, includefoot]{geometry}
\usepackage[colorlinks=true, pdfstartview=FitV, linkcolor=blue, citecolor=black, urlcolor=black,filecolor=magenta]{hyperref}
\usepackage{graphicx,verbatim,amsmath,amssymb,subfigure,enumerate,import}

\usepackage{tikz}
\usepackage{tikz-cd}
\usetikzlibrary{matrix}
\tikzstyle{vertex}=[circle]

\allowdisplaybreaks

\def\R{\mathbb{R}}
\def\N{\mathbb{N}}
\def\Z{\mathbb{Z}}

\def\meets{\dashv}

\def\Om{\Omega}

\def\Osub{\Omega_\varphi}
\def\Ohh{\Omega_\mathrm{HH}}
\def\Oa{\Omega_\mathrm{a}}
\def\Ost{\Omega_{\mathrm{ST}}}
\def\Oe{\Omega_\epsilon}

\def\iOa{\Omega_{\emph{a}}}

\def\RI{\textbf{R1}}
\def\RII{\textbf{R2}}

\def\iRI{\emph{\textbf{R1}}}

\renewcommand{\mid}{:}

\newcommand{\ch}{\mathrm{ch}}

\newtheorem{thm}{Theorem}[section]
\newtheorem{cor}[thm]{Corollary}

\newtheorem{lemma}[thm]{Lemma}
\newtheorem{prop}[thm]{Proposition}

\theoremstyle{definition}
\newtheorem{definition}[thm]{Definition}

\theoremstyle{remark}
\newtheorem{remark}[thm]{Remark}

\newtheorem*{Acknowledgements}{Acknowledgements}
\newtheorem{remarks}[thm]{Remarks}

\numberwithin{equation}{section}

 \title[An aperiodic tile with edge-to-edge orientational matching rules]{An aperiodic tile with edge-to-edge orientational matching rules}

\author{James J. Walton}
\address{James J. Walton,  School of Mathematical Sciences, University of Nottingham, University Park, 
Nottingham NG7 2RD, United Kingdom}
\email{james.walton@nottingham.ac.uk}
\author[Michael F. Whittaker]{Michael F. Whittaker}
\address{Michael F. Whittaker, School of Mathematics and Statistics, University of Glasgow, University Place, Glasgow Q12 8QQ, United Kingdom}
\email{Mike.Whittaker@glasgow.ac.uk}
 
 \thanks{This research was partially supported by EPSRC grant EP/R013691/1.}
 

\begin{document}

\begin{abstract}
We present a single, connected tile which can tile the plane but only non-periodically. The tile is hexagonal with edge markings, which impose simple rules as to how adjacent tiles are allowed to meet across edges. The first of these rules is a standard matching rule, that certain decorations match across edges. The second condition is a new type of matching rule, which allows tiles to meet only when certain decorations in a particular orientation are given the opposite charge. This forces the tiles to form a hierarchy of triangles, following a central idea of the Socolar--Taylor tilings. However, the new edge-to-edge orientational matching rule forces this structure in a very different way, which allows for a surprisingly simple proof of aperiodicity. We show that the hull of all tilings satisfying our rules is uniquely ergodic and that almost all tilings in the hull belong to a minimal core of tilings generated by substitution. Identifying tilings which are charge-flips of each other, these tilings are shown to have pure point dynamical spectrum and a regular model set structure.
\end{abstract}

\maketitle

\section{Introduction}

The fact that periodically arranged structures can be enforced by local rules is familiar to everyone. In covering the plane with unit squares so that squares must meet edge-to-edge, a periodic tessellation results. This simple principle of local constraints enforcing global structure explains how crystalline structures can form. Therefore, it was a great surprise to crystallographers in the 1980s when Dan Shechtman discovered a metal alloy whose diffraction pattern implied a great deal of structural order but had rotational symmetry precluding periodicity \cite{ShBlGrCa84}. Since the atomic organisation must still result from local interactions, the question arises of how such aperiodic patterns can result from only local rules. In the other direction, it is known that hierarchical aperiodic patterns generated by a substitution rule can be forced from local matching rules \cite{Goo98,Moz89,FerOll2010}.

Already in the 1960s, it had been observed by Robert Berger in solving Hao Wang's Domino Problem \cite{Wan61} that one may find square tiles with decorated edges that can tile the plane but only non-periodically \cite{Ber66, Rob71}. The first such set that he found had 20,426 tiles, initiating the hunt to find smaller aperiodic tile sets \cite{GruShe87}. The most famous, and arguably the most beautiful, such tile set is the pair of tiles discovered by Roger Penrose in the 1970s \cite{Pen79}. Isometric copies of these edge-decorated tiles (represented either as a pair of thick and thin rhomb, or as kite and dart tiles) can tile the plane but only aperiodically, and in fact form highly structured repetitive tilings with striking 10-fold rotational structure, similar to the rotational symmetry of the diffraction patterns of quasicrystals first observed by Shechtman.

Naturally, one wonders if two tiles are needed. The `Monotile Problem' asks: \emph{is there a single tile of the Euclidean plane for which copies of the tile can be used to tile the plane, but only non-periodically?} There are several ways to interpret this question. By `copies' of the tile one usually means isometric copies of the tile, through rotations, reflections and translations, although it is also of great interest to allow only rotations and translations \cite{GruShe87}. There are several demands one could make of such a `tile'. It is natural, for example, to ask that the tile does not have too wild a shape: it should be the closure of its interior, but one might also demand that it is a polytope, just a topological disc or perhaps merely that it has connected interior. And finally by `tiling the plane' one usually means that the tiles cover the plane but that distinct tiles overlap on at most their boundaries (however, we note here Gummelt's aperiodic tile that tiles the plane with overlaps \cite{Gum96}). One should also specify what rules are permitted in how tiles can be placed next to each other --- should these rules be forced by geometry alone, are colour matchings permitted, or can more complicated local rules be specified?

The best current solution to the monotile problem without overlapping tiles is the Socolar--Taylor tile \cite{SocTay11, SocTay12, Tay10}. This tile forces limit-periodic structures, closely related (but distinct) from the $(1+\epsilon+\epsilon^2)$-tilings \cite{PenTwstr,Pen95} and half-hex tilings \cite{BGG12}, using just a single tile. However, the matching rules restrict configurations not only of neighbouring tiles, but also non-touching next-nearest neighbours. An alternative form of the tile may be given which has nearest-neighbour locality, but at the cost that the tile has a complicated shape, with disconnected interior. So the hunt remains for a monotile with simple shape but also next-nearest neighbour matching rules.

\begin{figure}
	\def\svgwidth{\textwidth}
	\centering
	\includegraphics[scale = 0.3]{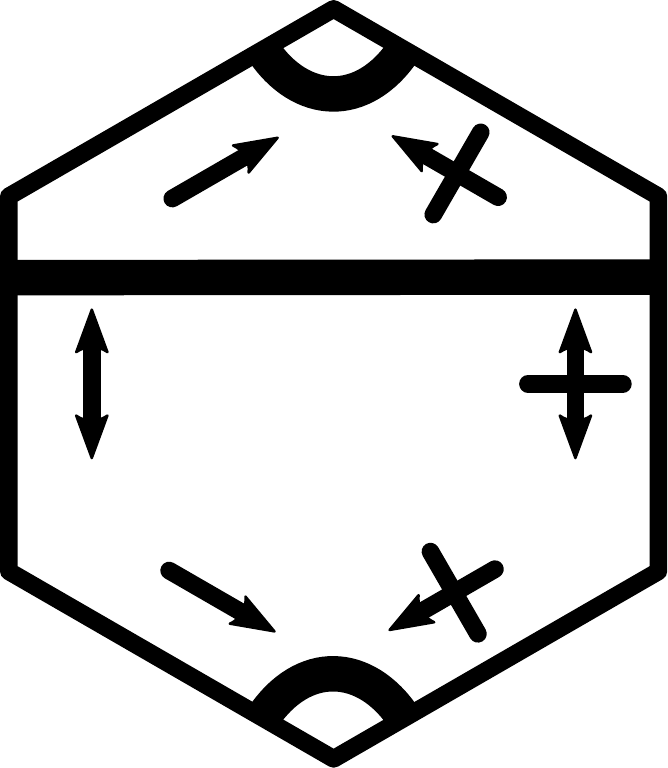}
	\caption{The tile. It consists of one horizontal straight \RI-segment and two \RI-turning segments, meeting the edges with the same offsets from the centre axes. The left-hand edges are labelled with negative \RII-charges, which are oriented from, top to bottom: clockwise, both clockwise and anticlockwise, and anticlockwise, respectively. The right-hand edges are labelled with positive \RII-charges, from top to bottom: anticlockwise, both clockwise and anticlockwise, and clockwise.}
	\label{fig: tile}
\end{figure}

In this paper we define a new aperiodic tile which, like the first form of the Socolar--Taylor tile, satisfies the requirement that it is a simple geometric shape, again being a hexagon. Unlike the Socolar--Taylor tile, the rules for which tiles are allowed to meet are nearest-neighbour, in fact only need to be checked on pairs of tiles meeting along an edge. The drawback is that whilst these rules are simply-stated and entirely local (in fact, edge-to-edge), they cannot be enforced by shape alone. Rather, whether two tiles can meet is determined by orientation as well as `charge' (equivalently decoration of one of two colours) along edges. Our tile is given in Figure \ref{fig: tile}. Two tiles $t_1$ and $t_2$ are permitted to meet along a shared edge $e$ only if:

\vbox{ 
\begin{itemize}
	\item[\RI]  the decorations of black lines of $t_1$ and $t_2$ are continuous across $e$;
	\item[\RII] whenever the two charges at $e$ in $t_1$ and $t_2$ both have a clockwise orientation, then they must be opposite in charge.
\end{itemize}
}

\begin{figure}
	\def\svgwidth{\textwidth}
	\centering
	\includegraphics[width = \textwidth]{patch.pdf}
	\caption{Patch of a valid tiling, where reflections of the tile of Figure \ref{fig: tile} are shaded in grey.}
	\label{fig: patch}
\end{figure}

Throughout, we shall call a tiling of the plane by isometric copies of the single tile of Figure \ref{fig: tile} \emph{valid} when tiles meet edge-to-edge and satisfy rules \RI\ and \RII\ at each edge. Our main theorem is that such tilings exist, and that they are always non-periodic:

\begin{thm}
\label{thm: main} There exist valid tilings by the tile of Figure \ref{fig: tile}. Moreover, any valid tiling $T$ is non-periodic; that is, if $T = T+x$ for $x \in \R^2$, then $x=0$.
\end{thm}

As we shall show in Section \ref{sec: hull}, the \RI\ rule alone forces \RI-triangles to contain a particular hierarchical nesting of others, just as in the Socolar--Taylor tilings. The way that arbitrarily large such triangles are forced with the new rule \RII, however, is quite different. Indeed, firstly, the matching rules being edge-to-edge allows for tilings with `infinite fault lines', as well as some others defects, as discussed in Section \ref{sec: hull}. Secondly, the patterns of tile parities (given by labelling hexagons only with the information of whether the tile of Figure \ref{fig: tile} or its mirror image is used) are very different, and in fact closely follow the structure of the \RI-edges, which are forced to carry the same parities across them.

This observation leads to a remarkably simple proof of aperiodicity, presented in full detail in Section \ref{sec:aperiodicity}, on pages \pageref{sec:aperiodicity}--\pageref{sec: existence}, but which we briefly outline now. It is easily seen that following one \RI-edge forwards to a second one, belonging to an exterior triangle, flips charges of the \RI-edges (Lemma \ref{lem: no chains}). This implies that the second \RI-edge is longer (Lemma \ref{lem: length}); the alternative would lead to a spiral of edges ending in a period three cycle (Figure \ref{fig: spiral}) resulting in a parity mismatch. Hence there is no upper limit on the lengths of the edges of \RI-triangles (Corollary \ref{cor: unbounded lines}), from which non-periodicity quickly follows.

Whilst the central objective of the paper is to showcase a tile admitting a novel and elementary proof of aperiodicity, in the second half of the paper we proceed to further analyse the collection of valid tilings and its dynamics. Analogously to the Socolar--Taylor tilings, there are a particular pair of `defect' tilings. The underlying \RI-decorations can be completed with charges to make a valid tiling in $8$ different ways, of which the two with $3$-fold rotational symmetry are the non-repetitive defects, each containing one instance of a particular vertex configuration. For the Socolar--Taylor tile, these defects are forced by the existence of such a vertex configuration, which does not appear in any other valid tiling. For our tile, this vertex in fact forces either this defect tiling, or alternatively what we call an `$n$-cycle'. There also exist tilings with `fault-lines', whose existence is linked to the fact that our matching rules are entirely edge-to-edge: two half-planes of valid tilings either side of a bi-infinite straight \RI-edge may, up to a charge-flip of one side, be pasted together along the \RI-edge. We give a classification result listing the possible \RI-decorations of valid tilings in Theorem \ref{thm: classification} which, like the proof of aperiodicity, is proven very directly from the simple structures of the \RI-edges, in particular the \RI-edge graph, introduced in Section \ref{sec: edge graphs}. There is a simple method of adding \RII\ charge decorations to the \RI-edge graph to make a valid tiling.

In Section \ref{sec: substitution} we introduce a substitution rule which generates valid tilings. Using the classification result above, one sees that defects can only ever appear sparsely and that all valid tilings have some form of supertile decomposition (Proposition \ref{prop: valid => supertile decomposition}). This allows us to deduce the existence of uniform patch frequencies, or equivalently unique ergodicity of the associated tiling dynamical system. With respect to the unique invariant measure, almost all tilings belong to the minimal core of substitution tilings. This minimal hull is a $2$-fold cover of another hull of tilings, also generated by substitution, given by identifying tilings which are charge-flips of each other. We are able to identify the multiplicity of the fibres of the map to the maximal equicontinuous factor (MEF), which demonstrates that the system is quite distinct from the Socolar-Taylor and Penrose $(1+\epsilon+\epsilon^2)$-tilings. The substitutional hull (modulo charge-flip) factors almost everywhere one-to-one to its MEF and so has pure point dynamical spectrum and the structure of a regular model set.

\begin{Acknowledgements}
We thank the anonymous referee for their helpful comments. We are grateful to Franz G\"{a}hler for discussions which helped to uncover the substitution generating a full measure subset of tilings in our hull.
\end{Acknowledgements}

\section{Aperiodicity}\label{sec:aperiodicity}

In this section we show that all valid tilings by the tile of Figure \ref{fig: tile} are non-periodic. We begin by defining the \RI-triangles and how one associates charges to their edges.

\subsection{\RI-Triangles}
The straight \RI-segments of the tiles are offset from the centre axis of the tile, so we may assign them a direction. We choose for them to point to the right when the \RI-segment is horizontal and offset towards the top of the tile, that is, when positioned as in Figure \ref{fig: tile}. The \emph{turns} in the \RI-lines (the small sections of decorations about two corners of the tile) are correspondingly offset, which means that they always turn leftwards from the direction of a straight \RI-segment leading into it. A maximal straight section of an \RI-line will be called an \emph{\RI-edge}. Since turns are always to the left, the \RI-edges always form either infinite lines (possibly composed of two edges, broken by a single turn) or triangles of three edges of the same length. With our convention of directing edges, triangles are always directed anticlockwise.

One may show more, namely that triangle edges must consist of $2^n - 1$ straight sections, where $n \in \N$, and that there is a hierarchical and identical formation of \RI-triangles inside every \RI-triangle of the same size. These observations won't be necessary for our proof, although this shall be proved in Section \ref{sec: hull} when we investigate the set of all possible valid tilings.

\subsection{Charges of triangle edges}
The region to the immediate left of a directed \RI-edge (even if it is infinite) is considered as the `inside' of the corresponding (possibly infinite) triangle. On a tile carrying an \RI-edge, precisely one clockwise oriented charge lies on the inside of the triangle, either positive or negative. We assign this charge also to the straight \RI-segment. So, for example, translates and rotates of the tile of Figure \ref{fig: tile} carry a negative charge and its reflection carries a positive charge. It is easy to see that two consecutive straight \RI-segments must be assigned the same charge, so we may consistently assign a charge $\ch(E) \in \{+,-\}$ to an entire \RI-edge $E$. Given a charge $c$ we let $c^*$ be its opposite, that is, $+^* = -$ and $-^* = +$.

Take an \RI-triangle edge $E_1$ that, following its orientation forwards, ends at a turn. The tile containing the turn carries a different \RI-edge $E_2$. In this case we say that \emph{$E_1$ leads to $E_2$} and write $E_1 \meets E_2$. We further specify that $E_1 \meets^N E_2$ if $E_2$ is offset \emph{near} to $E_1$ and that $E_1 \meets^F E_2$ if $E_2$ is offset \emph{far} from $E_1$. Equivalently, we have that $E_1 \meets^N E_2$ (resp.\ $E_1 \meets^F E_2$) if the triangle with edge $E_1$ is contained in (resp.\ is not contained in) the triangle with edge $E_2$; see Figure \ref{fig: charge transfer}.

\begin{figure}
	\def\svgscale{0.16}
	\centering
	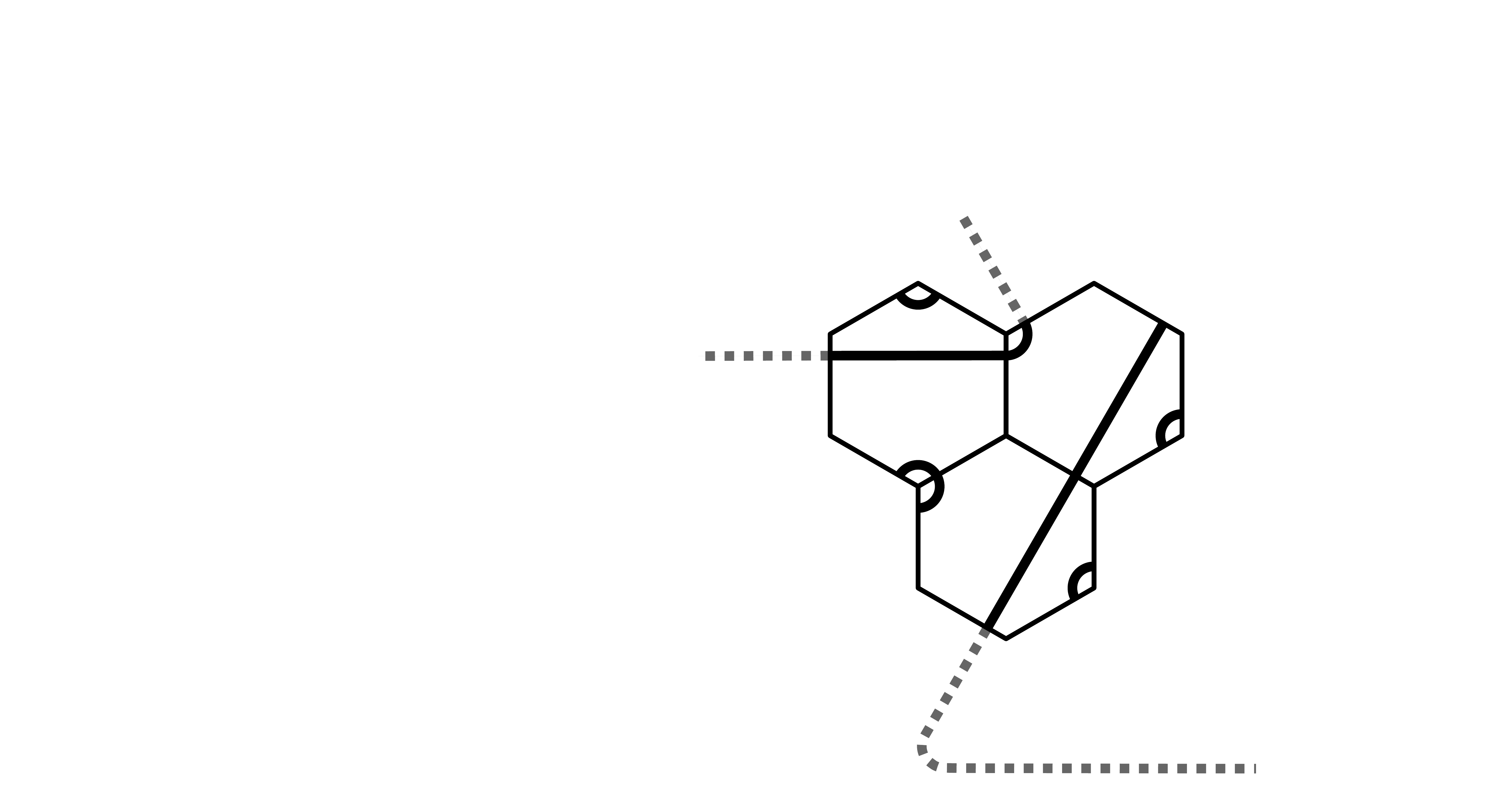
	\caption{The definition of $E_1 \meets^N E_2$ (left) and $E_1 \meets^F E_2$ (right). Relevant charges and tiles $t_1$ and $t_2$ are indicated as used in the proof of Lemma \ref{lem: no chains}.}
	\label{fig: charge transfer}
\end{figure}

\begin{lemma} \label{lem: no chains}
In any valid tiling, if $E_1 \meets^N E_2$ then $\ch(E_1) = \ch(E_2)$ and if $E_1 \meets^F E_2$ then $\ch(E_1) = \ch(E_2)^*$. In particular, there is no chain $E_1 \meets^F E_2 \meets^F E_3 \meets^F E_1$ of three edges.
\end{lemma}

\begin{proof} The proof follows from a simple inspection of Figure \ref{fig: charge transfer}. Indeed, suppose that $E_1 \meets^N E_2$. Let $t_1$ be the tile containing the final straight \RI-segment of $E_1$ before the turn and $t_2$ the tile containing the turn as well as a straight section of $E_2$. Let $e$ be the edge shared by the tiles $t_1$ and $t_2$. The charge on $e$ in $t_1$ is also clockwise oriented and equal to $c^*$, and the charge of $e$ in $t_2$ is clockwise oriented, and thus equal to $(c^*)^* = c$. By definition, this charge is equal to $\ch(E_2)$, as required. The case for $E_1 \meets^F E_2$ is analogous; in this case $\ch(E_2)$ is given by the charge of the edge opposite $e$ in $t_2$, which is $c^*$.

Given $E_1 \meets^F E_2 \meets^F E_3 \meets^F E_1$, by the above $\ch(E_1) = \ch(E_1)^{***} = \ch(E_1)^*$, a contradiction, so there is no such chain of three edges in a valid tiling. \end{proof}

\subsection{Finding edges of increasing length}

We let $L(E) \in \N \cup \{\infty\}$ be the length of an \RI-edge $E$, the number of tiles containing the straight segments of $E$ (so not including the turning tiles).

\begin{lemma} \label{lem: length}
Consider \RI-edges $E_1 \meets E_2$ in a valid tiling. Let $t$ denote the tile containing the terminating turn of $E_1$ and thus also a tile of $E_2$. Consider the collection $R$ of all tiles containing straight \RI-segments of $E_2$ starting from and including $t$ and heading to the right from $E_1$. Then $\# R = \infty$ if $L(E_1) = \infty$, and $\# R > L(E_1)$ otherwise.
\end{lemma}

\begin{figure}
	\def\svgscale{0.05}
	\centering
	\import{}{spiral.pdf_tex}
	\caption{Creating a spiral of edges from an edge $E_1$ (red tiles) containing more than or equal to the number of tiles of $R$ (blue tiles) in $E_2$, as in the proof of Lemma \ref{lem: length}.}
	\label{fig: spiral}
\end{figure}

\begin{proof}
Suppose, to the contrary, that $\#R < L(E_1) = \infty$ or that $\# R \leq L(E_1) < \infty$. We define a sequence $E_1$, $E_2$, $E_3$, \ldots of edges of respective triangles $\Delta_1$, $\Delta_2$, $\Delta_3$, \ldots, which we will show must spiral inwards and eventually form 3-periodic chain contradicting the previous lemma: see Figure \ref{fig: spiral}. The edges $E_1$ and $E_2$ are already as given and, having constructed $E_i$, we define $E_{i+1}$ by following $E_i$ rightwards from $E_{i-1}$ to its terminating turn, which is the tile containing $E_{i+1}$. Let $R_i$ be the collection of tiles containing straight \RI-segments of each $E_i$, starting from the tile with terminating turn of $E_{i-1}$ up to the terminating turn of $E_i$.

We observe that $E_1 \meets^F E_2$. Indeed, otherwise, the edge $E_3$ following $E_2$ would be part of the same \RI-triangle, in particular it would have length the same as that of $E_2$, causing it to intersect $E_1$, see the red dotted line in Figure \ref{fig: spiral}. Hence $E_2$ has orientation making $E_2 \meets E_3$. To prevent $E_3$ intersecting $E_1$, it is necessary that $\#R_3 \leq \#R_2 \leq L(E_2)$. So $E_2 \meets E_3$ and $\#R_3 \leq L(E_2) < \infty$. These properties of $E_2$ and $E_3$ are analogous to our initial assumptions on $E_1$ and $E_2$, so the argument repeats, showing that $E_i \meets^F E_{i+1}$ for all $i \in \N$ and $\#R_i$ is monotonically decreasing in $i$.

For $i \in \N$, the edge $E_{i+3}$ is parallel to $E_i$. All triangles $\Delta_j$ are in the exteriors of each other by the above, so we see that $\# R_i = \# R_{i+3}$ is only possible if $E_i = E_{i+3}$. Indeed, $R_{i+3}$ must belong to the triangular region bounded between $\Delta_i$, $\Delta_{i+1}$ and $\Delta_{i+2}$, which has strictly less than $\# R_i$ tiles in each row parallel to $E_i$, except for the row containing $E_i$ itself. Since $\# R_i$ can not strictly decrease indefinitely, we must have that $E_i = E_{i+3}$ for sufficiently large $i$. But this contradicts Lemma \ref{lem: no chains} since we have found a chain $E_i \meets^F E_{i+1} \meets^F E_{i+2} \meets^F E_i$.
\end{proof}

\begin{cor} \label{cor: unbounded lines}
In any valid tiling there is no finite upper bound on the length of \RI-edges.
\end{cor}

\begin{proof}
Supposing otherwise we may find a finite triangle of largest size, say with edge $E$. Then $E \meets E'$ for some edge $E'$, but the previous lemma implies that $L(E) < L(E')$, a contradiction. So there is either an infinite \RI-line or all triangles are finite but of unbounded size, as required.
\end{proof}

\begin{figure}
	\def\svgscale{0.1}
	\centering
	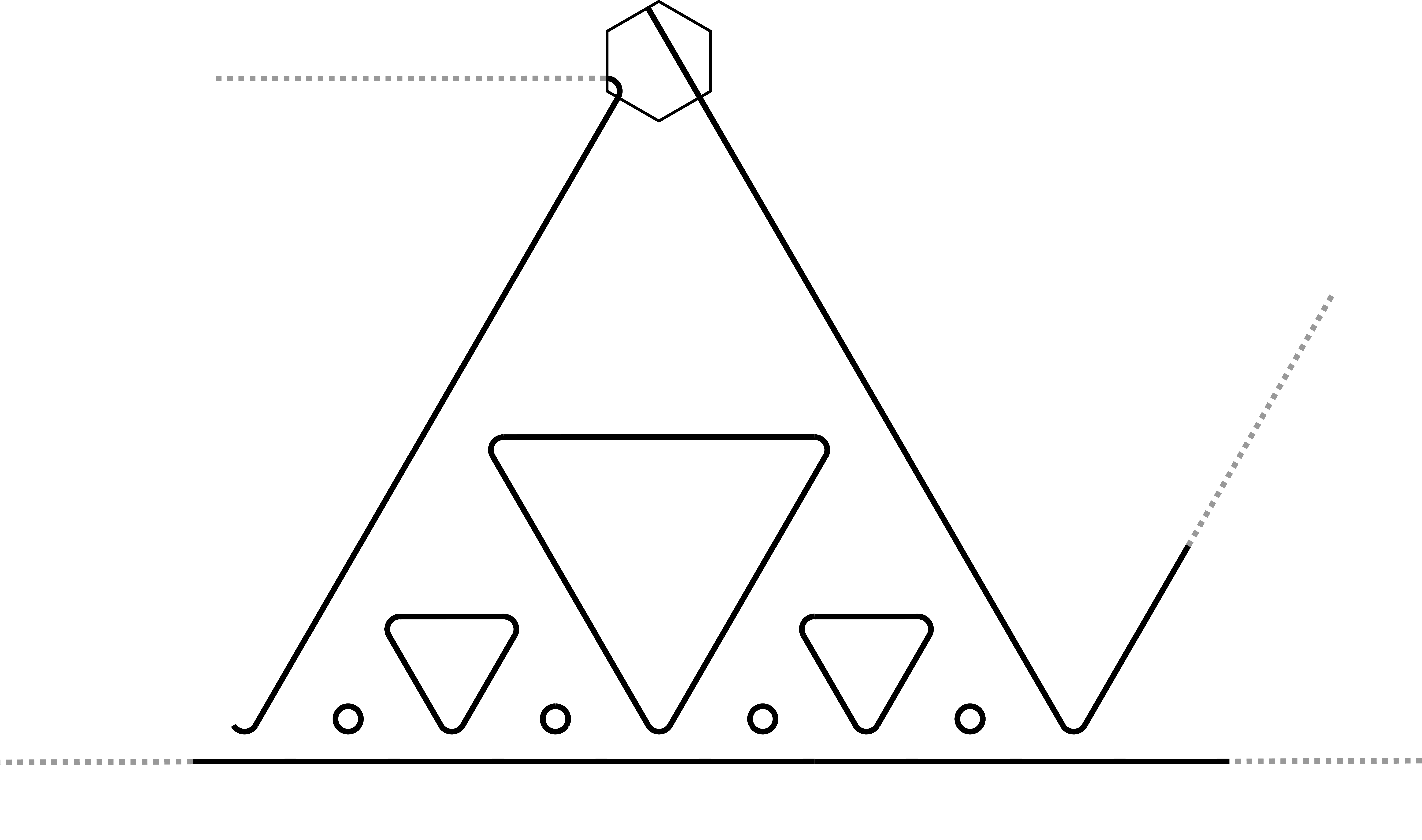
	\caption{Proof of aperiodicity in case of existence of infinite \RI-line $L$.}
	\label{fig: row of triangles}
\end{figure}

\begin{thm}\label{thm:nonperiodic}
Any valid tiling $T$ is non-periodic.
\end{thm}

\begin{proof}
By Corollary \ref{cor: unbounded lines} a valid tiling $T$ either contains an infinite \RI-line or triangles of arbitrarily large size. In the latter case, $T$ is non-periodic since any given translation will not be able to transfer sufficiently large triangles to others.

So we just need to show that any tiling $T$ with an infinite \RI-line is non-periodic. Assume that $T$ contains an infinite line $L$ with no turn. Orient the tiling so that $L$ points to the right and consider the set $A = \{\Delta_i \mid i \in \Z\}$ of triangles $\Delta_i$ that share turning tiles with $L$; see Figure \ref{fig: row of triangles}. Suppose that there is some $\Delta \in A$ of largest size. Supposing it is finite, follow its edge $E$ which heads upwards and to the right from $L$, leading to an edge $E'$ belonging to triangle $\Delta'$. The section of tiles from $t$ along $E'$ towards $L$ has more tiles than $L(E)$ by Lemma \ref{lem: length}, so in fact $\Delta'$ must reach a tile of $L$ and hence $\Delta' \in A$ too. But evidently $\Delta'$ is larger than $\Delta$, contradicting our assumption that $\Delta$ was the largest. So either $\Delta$ is infinite in size or there is no largest size of triangle in $A$. In either case $T$ cannot be periodic.

Suppose instead that $L$ has a turn. Let $E_1$ be the infinite length edge of $L$ directed towards the turn and $E_2$ the edge after the turn. We have $E_1 \meets E_3$ for another edge $E_3$, which is also infinite in length by Lemma \ref{lem: length}. We thus have three infinite \RI-edges and it is easy to see that any non-trivial translation would cause one to intersect another, non-parallel one, so $T$ is non-periodic.
\end{proof}

\section{Existence} \label{sec: existence}

\subsection{Standard \RI-tilings}

\begin{figure}
	\def\svgwidth{\textwidth}
	\centering
	\import{}{standard_patch.pdf_tex}
	\caption{Construction of standard patch $P_n$, here $P_3$. Starting with a triangle of size $2^n$, triangles of size $2^i$ are added for decreasing $i$ until ones of size $0$ are placed, defining $P_n$.}
	\label{fig: standard patch}
\end{figure}

We now prove that tilings satisfying \RI\ and \RII\ exist. We begin by giving a class of tilings with `standard' \RI\ decorations of hierarchically positioned \RI-triangles, which we later show can also be equipped with \RII\ decorations making valid tilings.

\begin{definition}
The \emph{size} $s(\Delta)$ of an \RI-triangle is defined to be $L(E) + 1$, where $E$ is an \RI-edge of $\Delta$. A loop of three \RI-turns is also considered to be an \RI-triangle, with size $1$.
\end{definition}

\begin{definition}
For each $n \in \N_0$, we define a \emph{standard patch} $P_n$ of hexagonal tiles, with (partial) \RI-decorations. We begin with an \RI-triangle $\Delta$ of size $2^n$ (see the left-hand patch of Figure \ref{fig: standard patch}). Another triangle $\Delta'$ of size $2^{n-1}$ is placed inside of $\Delta$ in the only way possible, that is, with edges leading to the centres of edges of $\Delta$ (see second patch of Figure \ref{fig: standard patch}). This leaves four triangular regions bounded between the edges of $\Delta$ and $\Delta'$. We repeat the above by placing four triangles of size $2^{n-2}$, one in each region with edges meeting the edges bounding the region (third patch of Figure \ref{fig: standard patch}). We continue until triangles of size $1$ are placed (fourth patch of Figure \ref{fig: standard patch}). The tiles carrying the \RI-triangle $\Delta$ and its interior, with \RI-decorations as constructed above, defines the patch $P_n$. 
\end{definition}

\begin{remark}
The above uniquely defines $P_n$, up to translation and rotation. Notice that all interior tiles of the patch $P_n$ are given full \RI-decorations, but the tiles meeting the boundary are only partially decorated.
\end{remark}

Although not needed for the arguments of this section, this hierarchical pattern of triangles is forced by the rule \RI, as we shall see in Proposition \ref{prop: R1 structure}. We now consider a natural collection of tilings associated to these standard patches:

\begin{definition} \label{def: standard tilings}
An \emph{\RI-tiling} is a tiling of hexagons decorated with \RI-decorations (but not \RII-decorations) which satisfies \RI. Such a tiling is called a \emph{standard \RI-tiling} if every finite patch is contained in a translate or rotate of a standard patch $P_n$. The collection of all standard \RI-tilings is denoted by $\Oa$.
\end{definition}

In the above definition one could allow $P_n$ to be extended to full \RI-decorations on boundary tiles (so that the given subpatch may share boundary tiles). However, this does not affect which \RI-tilings are standard. Indeed, a copy of $P_n$ is embedded in $P_m$ for any $m \geq n$, and in the interior of $P_m$ if $m \geq n+2$. Such embeddings also make it clear that standard \RI-tiilings exist, by choosing a nested union of such patches covering the entire plane.

\begin{remarks} \leavevmode
\begin{enumerate}
	\item We denote the collection of standard \RI-tilings by $\Oa$ since these tilings are easily seen to be mutually local derivable (MLD, see \cite{AObook}) to the {\bf \emph{arrowed}} \emph{hex tilings} \cite{BGG12}. Alternatively, these tilings are precisely those whose \RI-decorations come from the Socolar--Taylor tilings \cite{SocTay11}.
	\item These tilings can be constructed by substitution rules, either from the arrowed half hex substitution, via pseudo inflations (see \cite{BGG12}) or from an associated stone inflation on triangle tilings, as described in Section \ref{sec: substitution}.
\end{enumerate}
\end{remarks}

\subsection{\RI-edge graphs} \label{sec: edge graphs}
Here we shall see that all \RI-tilings satisfying certain restrictions on the structure of the \RI-edges (which includes any tiling in $\Oa$) may always be assigned charges so as to also satisfy \RII. The following lemma will be useful for this purpose, as well as later when we examine the hierarchical structure of valid tilings:

\begin{figure}
	\def\svgscale{0.2}
	\centering
	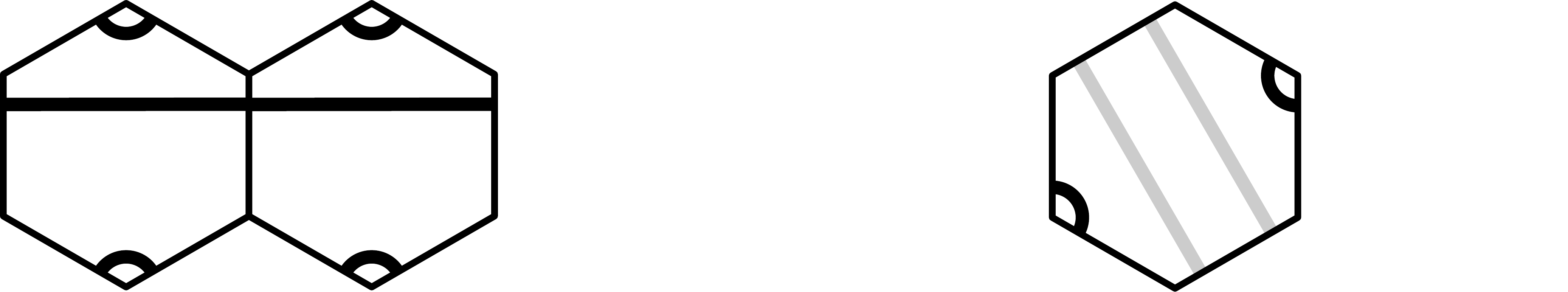
	\caption{Consistency of charges splits into three cases in the proof of Lemma \ref{lem: charge transfer sufficiency}: when two straight \RI-segments meet (left), when a straight segment meets a turn (middle) and when two turns meet (right). The two remaining cases of a straight segment meeting a turn are as in Figure \ref{fig: charge transfer}.}
	\label{fig: consistency}
\end{figure}

\begin{samepage}
\begin{lemma} \label{lem: charge transfer sufficiency}
Suppose that a tiling $T$ satisfies \RI\ and has charges assigned consistently across \RI-edges. Then $T$ is valid if and only if the following {\bf \emph{charge transfer property}} holds:
\begin{enumerate}
	\item if $E_1 \meets^N E_2$ then $\ch(E_1) = \ch(E_2)$;
	\item if $E_1 \meets^F E_2$ then $\ch(E_1) = \ch(E_2)^*$.
\end{enumerate}
\end{lemma}
\end{samepage}

\begin{proof}
By Lemma \ref{lem: no chains}, a valid tiling satisfies the charge transfer property. So suppose, conversely, that $T$ satisfies \RI\, has charges applied consistently across edges, and that the charge transfer property holds. We must show that \RII\ is also satisfied. Tiles can meet in one of three ways at an edge, either meeting either two, one or zero straight \RI-segments:
\begin{enumerate}
	\item If an edge meets two straight line segments (left of Figure \ref{fig: consistency}) then \RII\ is satisfied at this edge, since charges on \RI-edges are consistent.
	\item If a straight line meets a turn over an edge, either the turning tile is at the terminus or origin of the \RI-line. In the former case (middle of Figure \ref{fig: consistency}) there are no requirements for consistency of \RII. The latter case is depicted in Figure \ref{fig: charge transfer}; by Lemma \ref{lem: no chains} consistency is guaranteed by the charge transfer property.
	\item If two turns meet at the edge, \RII\ does not impose any restrictions, as seen on the right of Figure \ref{fig: consistency}.
\end{enumerate}
It follows that $T$ satisfies \RII\ and hence is valid, as required.
\end{proof}

Given a tiling satisfying \RI\, we construct an infinite directed graph $G$, called the \emph{\RI-edge graph}, from the \RI-edges by removing the turns, retaining orientations on \RI-edges and extending $E_1$ forwards to meet $E_2$ whenever $E_1 \meets E_2$, see Figure \ref{fig: edge graph}. We note that each path component of this graph has exactly the tree structure defined by the growth condition of the tilings in \cite{MamWhi19}. Indeed, every valid tiling in \cite{MamWhi19} defines a valid tiling under the rules here using the \RI-edge graph. However, the rules defined in \cite{MamWhi19} are not local in the sense of matching rules, and do not lead to a compact hull of tilings as they do in this paper (see the following section).

\begin{figure}
	\def\svgscale{0.08}
	\centering
	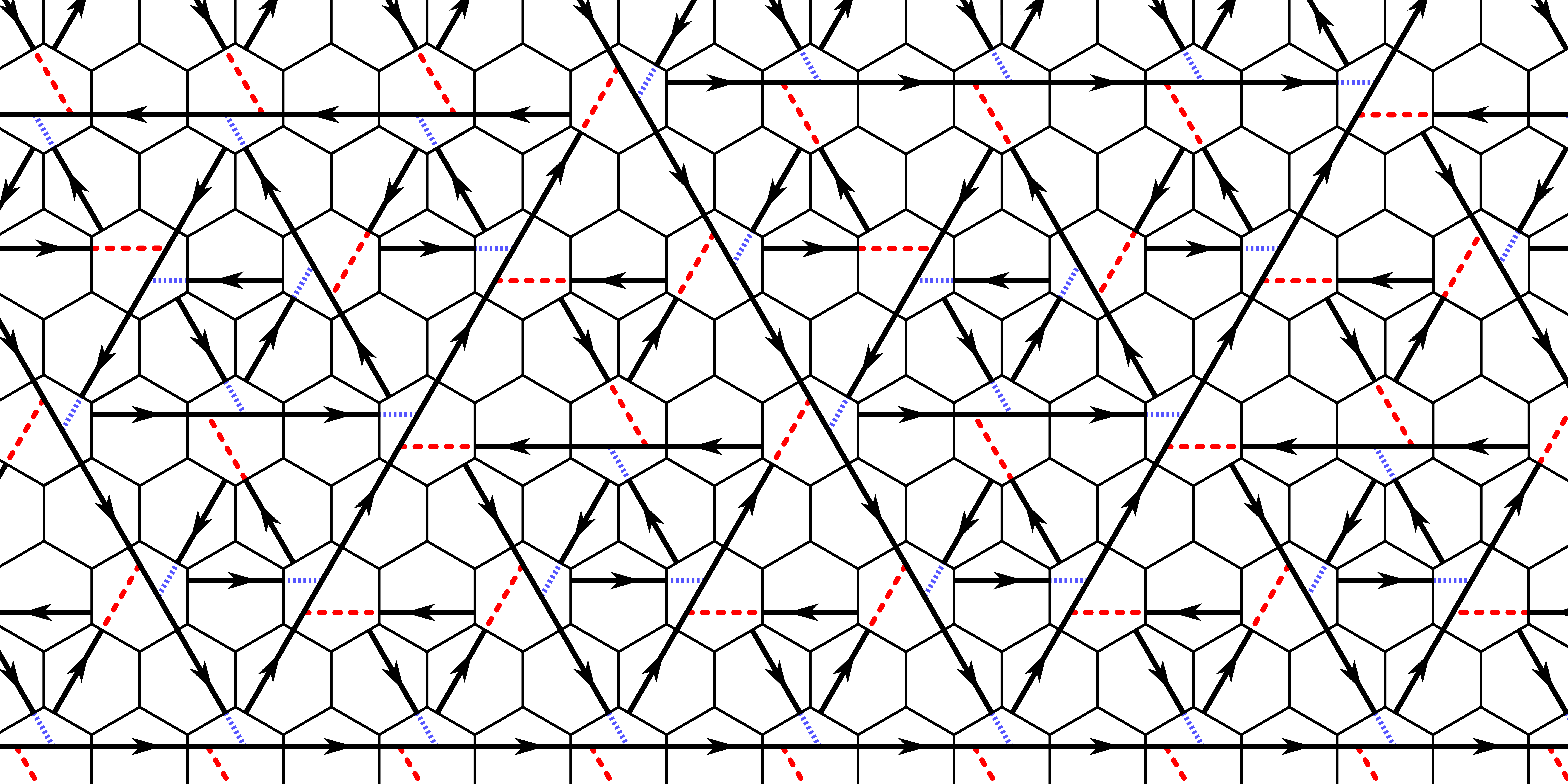
	\caption{The \RI-edge graph. Extensions of an edge $E_1$ to $E_2$ with $E_1 \meets^N E_2$ are given by dashed blue lines, those with $E_1 \meets^F E_2$ are red dotted lines.}
	\label{fig: edge graph}
\end{figure}

\begin{lemma} \label{lem: existence}
An \RI-tiling whose \RI-edge graph $G$ has no loops and $\ell$ connected components may be assigned charges in precisely $2^\ell$ different ways so as to give a valid tiling.
\end{lemma}

\begin{proof}
For each component of $G$, choose any \RI-edge and assign it one of the two possible charges. The charge transfer property forces the charges of the remaining edges of this component. Moreover, since this component is a tree, we cannot encounter inconsistencies in doing this. It follows that we may consistently apply charges \RI-edges in two ways to this component so as to satisfy the charge transfer property. Repeating for each component gives a consistent tiling by Lemma \ref{lem: charge transfer sufficiency}.
\end{proof}

The \RI-edge graph $G$ of any tiling $T \in \Oa$ has no loops. To see this, note that in any standard patch $P_n$, for any two (necessarily finite) edges $E_1$ and $E_2$ with $E_1 \meets E_2$ we have that $L(E_1) < L(E_2)$; this can be shown by induction, since $P_n$ is formed by embedding into the outer triangle of $P_n$ a copy of $P_{n-1}$ and three copies of its interior tiles. So the limiting tilings of $\Oa$ also have this property. Hence, since a directed path in $G$ can only take one from edges to longer edges, there can be no loops, and by the above may be consistently assigned charges to make valid tilings.

Combining the above existence result  with Theorem \ref{thm:nonperiodic} proves Theorem \ref{thm: main}, the main result of the paper.

\section{The hull of tilings} \label{sec: hull}

In this section we shall explain the structure of the \RI-triangles in valid tilings, which will allow us to classify the set of all valid tilings. We denote the set of all valid tilings by $\Om$. This set carries a natural topology (see \cite{Sad08} for an introduction to the topological study of aperiodic tilings). The space $\Om$ is sometimes called the \emph{(continuous) hull}. Belonging to the hull is a local (in fact, edge-to-edge) condition and as a consequence $\Om$ is a compact space.

The rule \RI\ alone limits the possible sizes of the \RI-triangles and the structure of \RI-triangles inside of them. For $n \in \N$ we define $a(n) = 2^i$ where $i \in \N_0$ is chosen as large as possible with $2^i$ dividing $n$. Recall that for an \RI-triangle $\Delta$ we denote its size by $s(\Delta) = L(E) + 1$, where $L(E)$ is the length of any of its \RI-edges.

\begin{figure}
	\def\svgwidth{\textwidth}
	\centering
	\import{}{triangle_sequence.pdf_tex}
	\caption{Proof of Proposition \ref{prop: R1 structure}, for $k = 16$ (the 16 tiles up and to the right of $t$ are not all indicated). The \RI-edges to the right and top-right of $t$ are allowed to be offset in two possible ways, so these lines are marked in grey. Notice that the sizes of triangles are, respectively, 1, 2, 1, 4, 1, 2, 1, 8, 1, 2, 1, 4, 1, 2, 1 and finally $s(\Delta_{16}) = 16$: induction using the right-edge of $\Delta_8$ and the edge between $t_8$ and $t_{16}$ forces the blue triangles, which force the line meeting $t_{16}$.}
	\label{fig: triangle sequence} 
\end{figure}

\begin{prop} \label{prop: R1 structure}
Suppose a tiling $T$ satisfying \RI\ is given. For any \RI-triangles $\Delta$ in $T$ we have that $s(\Delta) = 2^n$ for some $n \in \N \cup \{0\}$. Moreover, suppose that $E_1$ and $E_2$ are \RI-edges leading to or from a tile $t$ which, without loss of generality, are positioned so that $E_1$ is horizontal and $E_2$ extends up and to the right from $t$ (see Figure \ref{fig: triangle sequence}). Suppose that there are at least $k$ tiles from the right of (and not including) $t$ along $E_1$, and similarly $k$ tiles from $t$ up and to the right along $E_2$. Consider the \RI-triangles $\Delta_1$, $\Delta_2$,\ldots, which share a tile with $E_1$, naturally ordered by where they meet $E_1$ from left to right. Then $s(\Delta_n) = a(n)$ for all $n = 1,\ldots,k$.
\end{prop}

\begin{proof}
Beginning at the left, for any $k$ we have $s(\Delta_1)=1$, since the two turns of $\Delta_1$ in $E_1$ and $E_2$ are already connected without straight edges. Suppose now that the result above on the sizes $s(\Delta_i)$ holds for all $k < N$, and that for all \RI-triangles $\Delta$ with $s(\Delta) < N$ we have that $s(\Delta) = 2^n$ for some $n \in \N \cup \{0\}$.

We claim that if there is a triangle of size $N$, then $N$ is a power of $2$. Indeed, let $n \in \N$ be such that $2^{n-1} < N \leq 2^n$. Take two edges $E_1$ and $E_2$ of $\Delta$ and note that they are positioned as in the statement of the proposition, so by induction we have that $s(\Delta_{2^{n-1}}) = 2^{n-1}$. We see that $\Delta_{2^{n-1}}$ will only fit inside of $\Delta$ if $s(\Delta) \geq 2^n$, hence $s(\Delta) = 2^n$, as required.

Next we show that for $E_1$ and $E_2$ as in the statement of the proposition, $s(\Delta_N) = a(N)$. Take $n \in \N$ such that $2^{n-1} < N \leq 2^n$. Let us write $t_0 = t$, and $t_i$ to be the $i^{\text{th}}$ tile from $t$ along $E_1$, that is, the tile shared with $\Delta_i$. If $N < 2^n$ then by induction on the edge $E_1$ and the right-hand edge of $\Delta_{2^{n-1}}$ we see that $s(\Delta_N) = a(N - 2^{n-1}) = a(N)$, as required. Indeed, the tile $t_N$ is distance $N-2^{n-1} < 2^{n-1}$ from $t_{2^{n-1}}$, and $s(\Delta_{2^{n-1}}) = 2^{n-1}$, so the right-hand side of $\Delta_{2^{n-1}}$ is long enough to force $\Delta_N$.

Finally, suppose that $N = 2^n$; we wish to show that $s(\Delta_N) = 2^n$. Consider the collection $R$ of tiles heading in a straight line from $t_N$ up and to the left, terminating at $E_2$ (those in view in Figure \ref{fig: triangle sequence} are shaded in grey, with $N=16$). Notice that the middle tile $s \in R$ contains the top right turn of $\Delta_{2^{n-1}}$, so the straight \RI-segment of this $s$ is parallel to the row of tiles $R$.\footnote{Notice that if \RII\ is also satisfied, we may conclude already that the \RI-edge $E$ passing through $s$ has length at least $L(E) > 2^{n-1}$ by Lemma \ref{lem: length}, and so extends all the way to $t_N$. Since there are no triangles of size between $2^{n-1}$ and $2^n$ by induction, we may thus already conclude that $s(\Delta_N) = N$. However, as is shown, \RII\ is not needed here.} Similarly, the triangle $\Delta_{(2^{n-1} + 2^{n-2})}$ (given by $\Delta_{12}$ in Figure \ref{fig: triangle sequence}) has top right corner in a tile of $R$ lying half-way between $t_N$ and $s$, so the straight \RI-segment of this tile is also parallel to $R$. We may repeat this to see that each tile in $R$ between $s$ and $t_N$ has a straight line segment running parallel to the direction of $R$, so these segments must be contiguous and form part of an \RI-edge running at least between $s$ to $t_N$. This \RI-line is already composed of $2^{n-1}$ straight sections, and so by the first part of the proof, restricting the sizes of triangles, it must in fact be of length at least $2^n-1$ (alternatively, we could repeat the previous argument on the triangular region between $E_2$ and the top edge of $\Delta_{2^{n-1}}$). It cannot be longer, or else it would pass through $E_1$ or $E_2$, so we conclude that $s(\Delta_N) = 2^n = a(N)$, as required.
\end{proof}

Notice that for $k = 2^n$ we have that $s(\Delta_{k/2}) = k/2$. So for any \RI-triangle (or more generally, a triangular region bounded by \RI-edges), Proposition \ref{prop: R1 structure} shows that \RI\ alone forces an \RI-triangle of half the size positioned with corners at the midpoint of the bounding edges. So the hierarchical pattern of triangles given in constructing the standard patterns in Section \ref{sec: existence} is forced by \RI.

Given a tiling $T \in \Om$, let $\RI(T)$ be the associated \emph{\RI-tiling} given by forgetting the \RII-decorations. Similarly, we let $\RI(\Om) = \{\RI(T) \mid T \in \Om\}$. We now determine which tilings can belong to $\RI(\Om)$, which can be split into the following classes:
\begin{enumerate}
	\item There is no infinite \RI-line and:
		\begin{enumerate}
			\item every \RI-triangle is contained in infinitely many others or
			\item every \RI-triangle is contained in only finitely many others.
		\end{enumerate}
	\item There is at least one infinite \RI-line where:		
		\begin{enumerate}
			\item there is an infinite \RI-line containing no turn;
			\item every infinite \RI-line contains a turn.
		\end{enumerate}
\end{enumerate}

Recall from Definition \ref{def: standard tilings} that $\Oa$ is the collection of tilings whose finite patches are contained in translates of the standard patches $P_n$ given in Section \ref{sec: existence}.

\begin{lemma}
If $T \in \iRI(\Om)$ and is in the class 1a, then $T \in \iOa$.
\end{lemma}

\begin{proof}
Since every triangle is contained in a larger one, we may construct an infinite nested sequence of triangles $\Delta_1$, $\Delta_2$, \ldots whose interiors cover the entire plane. By Proposition \ref{prop: R1 structure}, the interiors of these triangles are forced to have standard \RI-decorations, so $T \in \Oa$, as required.
\end{proof}

\begin{lemma}
If $T \in \iRI(\Om)$ and is in the class 1b, then $T \in \iOa$.
\end{lemma}

\begin{proof}
Consider an `outer \RI-triangle' $\Delta$ of $T$, that is, one which is not contained in any other triangle. Take an edge $E_1$ of $\Delta$ and follow the sequence of edges $E_1 \meets E_2 \meets E_3 \meets \cdots$. Since the triangle containing $E_k$ is exterior to that containing $E_{k+1}$, we have that $E_k \meets^F E_{k+1}$ for all $k$. Analogously to the proof of Lemma \ref{lem: length}, we thus construct a spiral of edges whose triangles are all exterior to each other. Lemma \ref{lem: length} implies that the edges become longer, infinitely spiralling outwards. Since all triangles are exterior to each other, we see that arbitrarily large patches about the initial edge $E_1$ are contained in the triangular regions bounded by $E_k$ and $E_{k+1}$. But the pattern of \RI-triangles in such regions are forced to be standard ones also found between edges of \RI-triangles, by Proposition \ref{prop: R1 structure}. So arbitrarily large patches about $E_1$ are patches of tilings of $\Oa$, so $T \in \Oa$, as required.
\end{proof}

In summary, those generic tilings $T \in \RI(\Om)$ without an infinite \RI-line are also in $\Oa$. By Lemma \ref{lem: existence} and the remarks following it, we may choose compatible \RII-decorations for such tilings to give valid tilings. The edge graph is typically, but not always, connected (see Lemma \ref{lem: connected components}).

\begin{lemma}
Suppose that $T \in \iRI(\Om)$ and is in the class 2a. Let $L$ be the unique infinite \RI-line without turn. Then $T$ is a union of two partial tilings $T_1$ and $T_2$, where $T_1$ covers a half-plane one side of $L$ and $T_2$ the other, and where each $T_i$ is a subset of tiles from some $T_i' \in \iOa$.
\end{lemma}

\begin{proof}
Consider the pattern of \RI-triangles one side of $L$, and take any triangle $\Delta$ in it with turn on a tile carrying $L$. The right-hand edge of $\Delta = \Delta_1$ leads to the left-hand edge of another triangle $\Delta_2$, which also shares a turn with $L$ by Lemma \ref{lem: length}. Repeating this we construct a sequence of triangles $(\Delta_i)_{i \in \N}$ which meet tiles of $L$. Note that by Lemma \ref{lem: length} and Proposition \ref{prop: R1 structure}, we have that $s(\Delta_{i+1}) \geq 2 s(\Delta_i)$ and then the triangles $\Delta_i$ force standard \RI-decorations in increasingly large regions in the half-plane. More precisely, the union of the inside of $\Delta_i$, the triangular region between its left edge and $L$ and the other triangular region between its right-edge and $L$ are forced to be a standard decoration seen in some $P_n$. These forced standard patches either cover the whole half space, or $s(\Delta_{i+1}) = 2s(\Delta_i)$ for sufficiently large $i$. In the latter case, an analogous argument to that in the final paragraph of the proof of Proposition \ref{prop: R1 structure} implies the existence of an infinite \RI-triangle meeting $L$ to the left of $\Delta$. Again by Proposition \ref{prop: R1 structure}, this forces a standard \RI-decoration on this half-plane.
\end{proof}

For a tiling $T$ as above with $T \notin \Oa$, we call $T$ a tiling with an \emph{infinite fault line}. For tilings $T \in \Oa$, a triangle meeting $L$ has the same size as the triangle opposite it across $L$. But since our matching rules are edge-to-edge, one may freely shift the tiles of one half-plane relative to the other half. Again, by Lemma \ref{lem: existence}, any such tiling can be equipped with compatible \RII-decorations and so is an element of $\RI(\Om)$.

\begin{figure}
	\def\svgscale{0.07}
	\centering
	\import{}{cycle.pdf_tex}
	\caption{In case 2b, the three infinite lines with turns are forced to arrange themselves as a cycle. The left-hand picture shows a configuration which cannot occur. The right-hand picture, in this case an infinite $1$-cycle, can occur.}
	\label{fig: cycle}
\end{figure}

Finally, suppose that we are in Case 2b. Consider the infinite \RI-edges with turn $L_i$. By Lemma \ref{lem: length}, the edge of each $L_i$ heading towards the turn leads to another infinite edge. Then the only possibility is for there to be three infinite \RI-edges, $L_1$, $L_2$ and $L_3$, arranged in a cyclic fashion as in Figure \ref{fig: cycle}. Let $E_i$ be the edge of $L_i$ leading to the turn and $E_i'$ be the other edge of $L_i$. There are two cases: either $E_i \meets^F E_{i+1}$ or $E_i \meets^F E_{i+1}'$ for all $i$ (considered modulo $3$, and with the $L_i$ ordered appropriately). The first case (left of Figure \ref{fig: cycle}) is ruled out by Lemma \ref{lem: no chains}. The second possibility (right of Figure \ref{fig: cycle}), however, is possible. By Proposition \ref{prop: R1 structure}, the pattern of \RI-triangles is completely determined by the number of tiles of $E_i'$ starting from the tile after that containing the turn of $L_i$ up to and including the tile containing the turning tile of $E_{i-1}$, and that there are $2^n$ such tiles for some $n \in \N \cup \{0\}$. We call such a tiling an \emph{infinite $n$-cycle}. Up to translation, for each $n \in \N$ there are precisely two such \RI-tilings (which are related by a rotation) by Proposition \ref{prop: R1 structure}. We have thus proved the following:

\begin{thm} \label{thm: classification}
If $T \in \iRI(\Om)$ then either:
\begin{enumerate}
	\item $T \in \iOa$;
	\item $T$ is a tiling with infinite fault line;
	\item $T$ is an infinite $n$-cycle tiling.
\end{enumerate}
In each case, there are $2^\ell$ tilings $T' \in \Om$ with $\iRI(T') = T$, where $\ell$ is the number of path components of the \RI-edge graph of $T$.
\end{thm}

Generically $\ell = 1$, but there are also cases where $\ell = 2$ or $3$ in Case 1 of the above theorem (see Lemma \ref{lem: connected components}). In Case 2, $\ell = 1$, $2$ or $3$, depending on if there are $0$, $1$ or $2$ infinite triangles meeting the fault line, respectively. In Case 3, $\ell = 3$.

\section{Generating valid tilings by substitution} \label{sec: substitution}

In this section we shall build a primitive and recognisable substitution rule which generates valid tilings. This will allow us to analyse frequencies of patches and deduce unique ergodicity of the hull of all valid tilings. Measure-theoretically, almost all valid tilings will belong to the minimal component of substitution tilings.

\subsection{MLD triangle tilings}
For a `stone inflation' \cite{AObook}, inflates of tiles are precisely tiled by copies of the originals. Since one may not tile the hexagon with smaller hexagons, we cannot define a stone inflation for our original tiles, although one may define a `substitution with overlaps' (like the Penrose kite and dart or rhomb substitution \cite{Pen79}) or a `pseudo-inflation' \cite{AObook, BGG12}.

We choose instead to pass to the dual triangle tiling, which will allow us to define a stone inflation that generates tilings that are MLD to valid tilings. Hence our prototile set will consist of decorated equilateral triangles. Each triangle has edges labelled with:
\begin{itemize}
	\item An arrow at each edge's centre, pointing in one of two possible directions perpendicular to the edge.
	\item A charge, either $+$ or $-$.
\end{itemize}
So each triangle edge has one of $4$ possible labellings. The direction is assigned by the direction of offset towards the \RI-line which crosses the corresponding hex-tile edge (see Figure \ref{fig: MLD}). If this arrow is pointed towards a straight section of an \RI-edge, then we assign it the same charge as this edge. Otherwise, the arrow is pointed towards the start of a left-turn, for which we may determine a charge analogously to the charge transfer property of Lemma \ref{lem: charge transfer sufficiency}. That is, if the turn $t$ is part of a hex-tile with straight \RI-segment $s$ of charge $c$, then we also label our triangle edge with charge $c$ if $s$ is offset near to $t$. Otherwise, if $s$ is offset far from $t$, then we assign our triangle edge charge $c^*$.

It is easily seen that a valid tiling by the tile of Figure \ref{fig: tile} determines, by a local-defined rule, a triangle tiling and vice versa, so two such are MLD. Unless stated otherwise, we shall assume in this section that all tilings have been converted to triangle tilings, and call such tilings \emph{valid} if their corresponding hex tilings are.

\begin{figure}
	\def\svgscale{0.15}
	\centering
	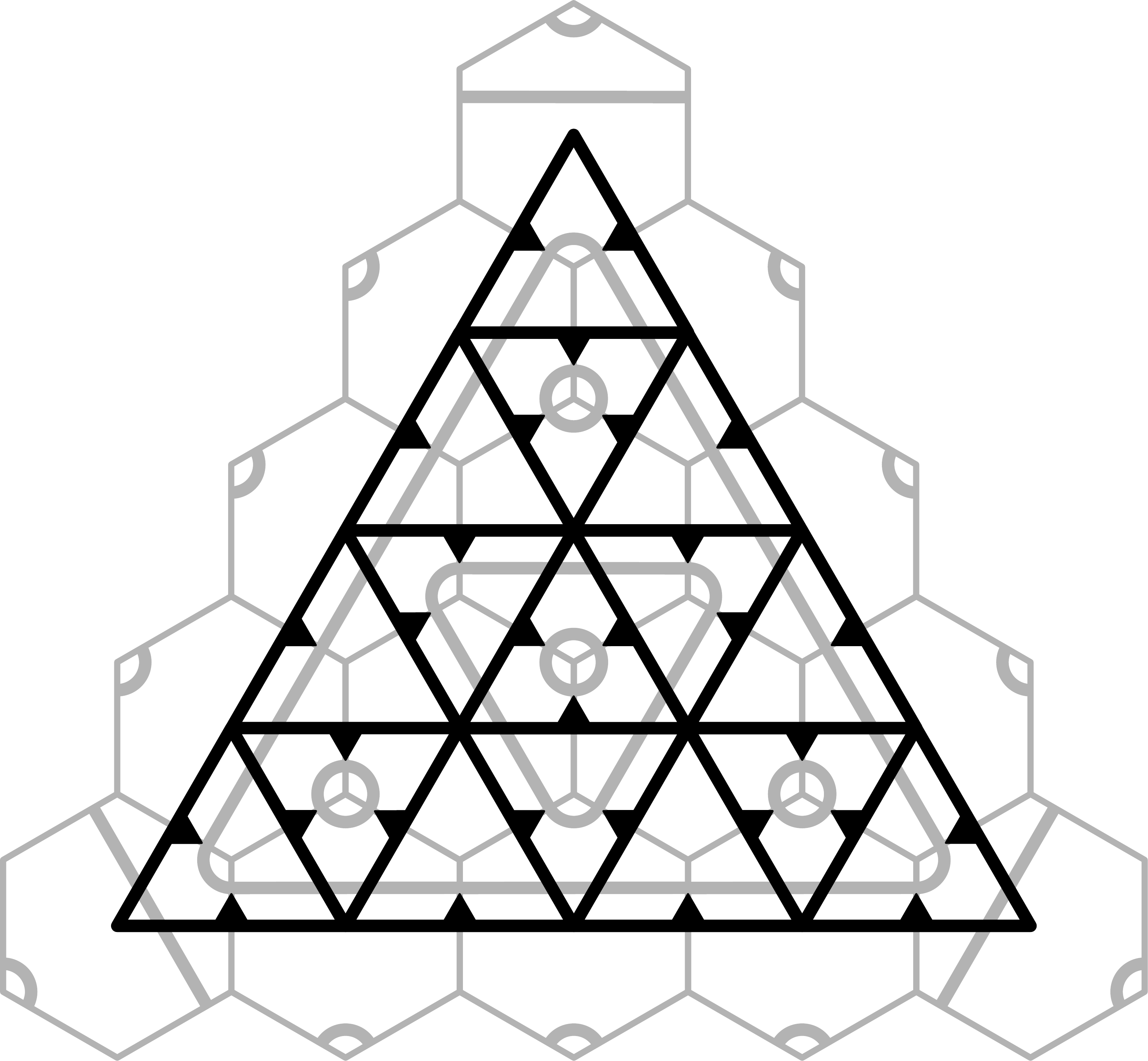
	\caption{The MLD equivalence on small patches of a valid tiling and a tiling admitted by $\varphi$, with charges removed for clarity.}
	\label{fig: MLD}
\end{figure}

\subsection{The substitution}
Given a decoration of \RI-triangles from a valid tiling, there is a natural way of associating to it another whose level $n \geq 1$ triangles correspond to those of level $(n-1)$ from the original. Simply double the size of each triangle, preserving the charges of their edges, and add new level $0$ triangles in the gaps. By Lemma \ref{lem: charge transfer sufficiency}, this gives the \RI-decorations of another valid tiling. The procedure defines our stone inflation on triangle tiles: each triangle is inflated by a factor of $2$, split into $4$ triangles of the original size, the labels on the outside edges are preserved and the central triangle has all edges directed inwards (which corresponds to adding a level-$0$ triangle). The charges of the internal edges are the unique ones which satisfy the charge transfer property of Lemma \ref{lem: charge transfer sufficiency} upon further substitutions, or equivalently how charges are assigned in the MLD relation described above. In the triangle tilings, a straight section of edges labelled with the same directions corresponds to an \RI-edge, whilst a vertex of a triangle between two inwards pointing edges corresponds to an \RI-turn.

This defines a tiling substitution $\varphi$ of finite local complexity \cite{AObook} on the set of all edge-labelled triangles (see Figure \ref{fig: sub}). As usual, we may extend $\varphi$ to act on (possibly infinite) patches. Recall that $\varphi$ is called \emph{primitive} if there exists some $N \in \N$ so that $\varphi^N(t)$ contains a copy of each tile for any prototile $t$. A primitive substitution rule is defined by the $6$ prototiles in Figure \ref{fig: sub}, along with their rotates and charge-flips. Since each tile is asymmetric, this gives a prototile set $\mathcal{P}$ of $6 \times 6 \times 2 = 72$ tiles.

A tiling is said to be \emph{admitted by the substitution} $\varphi$ if every finite patch is contained in a translate of an \emph{$n$-supertile}, which is a patch $\varphi^n(p)$ for some $n \in \N$ and $p \in \mathcal{P}$. We also refer to a $1$-supertile as simply a \emph{supertile}. We denote by $\Osub$ the set of all tilings admitted by $\varphi$.

For tilings $T$ and $T'$ with $T = \varphi^n(T')$, we call $T'$ a \emph{level-$n$ supertiling} of $T$, and a tile $t \in T'$ a \emph{level-$n$ supertile}. By primitivity, tilings admitted by $\varphi$ exist and have admitted supertilings of all levels. When the supertiling is always uniquely defined by $T$ (and then necessarily MLD to it), we call $\varphi$ \emph{recognisable}.

\begin{figure}
	\def\svgscale{0.15}
	\centering
	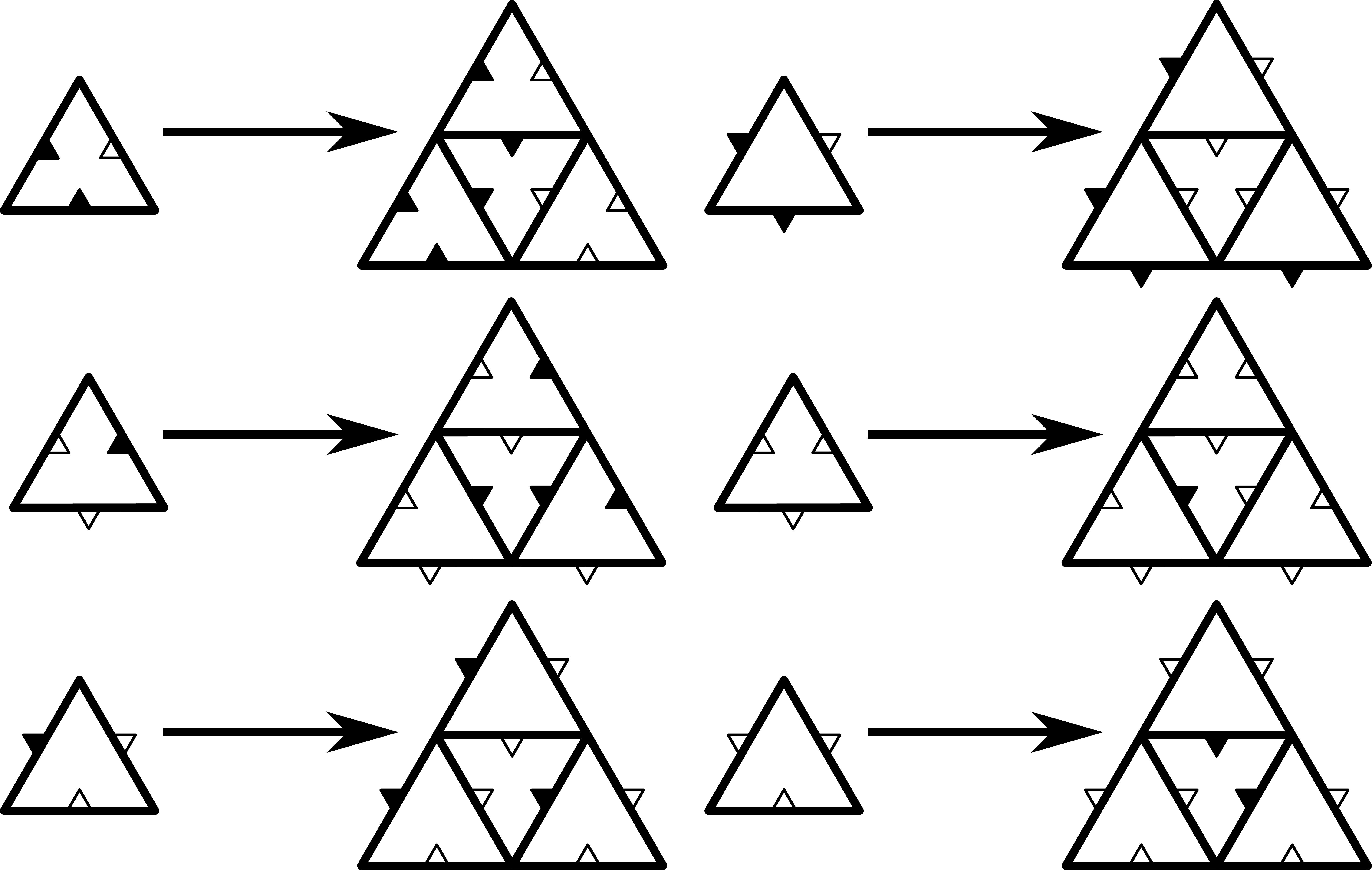
	\caption{The tiling substitution $\varphi$. Positively charged edges are indicated with black arrows, negatively charged edges have white arrows.}
	\label{fig: sub}
\end{figure}

\begin{lemma} \label{lem: primitive and recognisable}
The substitution $\varphi$ on the prototile set $\mathcal{P}$ is primitive and recognisable.
\end{lemma}

\begin{proof}
Primitivity may be checked by hand, or more quickly by a computer\footnote{We thank Franz G\"{a}hler for assisting us with this.}. Let $T$ be an admitted tiling, with associated supertiling $T'$ (which we scale here by a factor of $2$). Notice that under substitution only the central triangle of a supertile has all edges directed inwards. So the tiles of $T'$ are located periodically, with centres at precisely the locations of the all-inward directed triangles of $T$ (which correspond \RI-triangles of size $0$). There is also no choice for the labellings on the exteriors of supertiles, they must agree with the edges they overlap in $T$, since labels of exterior edges are kept constant under substitution. So the supertiling $T'$ is locally determined and uniquely defined, meaning that $\varphi$ is recognisable.
\end{proof}

\begin{lemma}
A tiling admitted by $\varphi$ is valid, that is $\Osub \subset \Om$.
\end{lemma}

\begin{proof}
By the definition of the substitution rule, an $n$-supertile $\varphi^n(p)$, for $p \in \mathcal{P}$, has charges assigned consistently to \RI-edges (since edges of length $2^n$ are doubled in length under substitution and given the same charge). The charge transfer property of Section \ref{sec: edge graphs} also holds, by the way that internal edges are assigned charges in the definition of $\varphi$. Since a tiling admitted by $\varphi$ can be covered by such patches, the corresponding hex tiling is valid, by Lemma \ref{lem: charge transfer sufficiency}.
\end{proof}

\subsection{Supertile decompositions of valid tilings}

Not all valid tilings in $\Om$ are admitted by $\varphi$. For example, tilings containing $n$-cycles or fault lines are not. However, we shall see that every valid tiling has some form of decomposition into supertiles of all levels. This will allow us to deduce that all valid tilings have uniform patch frequencies, even if they are not admitted by $\varphi$.

\begin{figure}
	\def\svgscale{0.03}
	\centering
	\import{}{omitted.pdf_tex}
	\caption{Tiles omitted from prototile set $\mathcal{P}$ are given by the above three classes, where we take all charge-flips and rotates. The third class may have right-hand edge labelled freely.}
	\label{fig: omitted}
\end{figure}

Recall that in defining the tiling substitution $\varphi$, we kept only $72$ of the $2 \times (4^3) = 128$ possible decorated triangles for our prototile set. We omit precisely the tiles of Figure \ref{fig: omitted}, along with their rotates and charge-flips. The first two are preserved under rotation by $2\pi/6$, so together account for $8$ tiles. The third class is asymmetric and may have right-hand direction and charge assigned freely, so these account for the remaining $2 \times 4 \times 6 = 48$ tiles. We define $\mathcal{P}'$ to be the union of prototiles $\mathcal{P}$ and the $8$ tiles in the first and second classes of Figure \ref{fig: omitted}. Then $\varphi$ defines also a (non-primitive) substitution on $\mathcal{P}'$.

\begin{lemma} \label{lem: triangles => supertiles}
A triangular region $P$ of a valid tiling bounded between three \RI-edges is given by $\varphi^n(t)$ for $t \in \mathcal{P}'$.
\end{lemma}

\begin{proof}
The $3$ outer edges of $P$ are consistently directed and charged. Let $t$ be the triangular tile with corresponding labels of its edges. Then $t$ cannot belong to the third class of tile in Figure \ref{fig: omitted}. Indeed, in that case the left edge necessarily has to continue downwards from $P$; it cannot turn towards the bottom edge since that is directed downwards. This then violates the charge transfer property with the bottom edge.

By Proposition \ref{prop: R1 structure}, the interior of $P$ has \RI-decorations determined by the size of $P$, which must have edge lengths a power of $2$. So the \RI-decorations of $P$ agree with $\varphi^n(t)$. By construction, the outer edges of $P$ have charges agreeing with $\varphi^n(t)$. Charges of interior \RI-edges of $P$ are determined by the charge transfer property of Lemma \ref{lem: charge transfer sufficiency}. By the construction of $\varphi$, the charge transfer property holds on $\varphi^n(t)$ too, so the two patches agree.
\end{proof}

The proposition below exhibits the extent to which valid tilings may be equipped with supertile decompositions. There is an interesting singular case, which has an analogous occurrence in the Socolar--Taylor tilings. Notice that a rotated copy of $P_n$ may be found at the centre of $P_{n+1}$ (see Figure \ref{fig: standard patch}). So there is an \RI-tiling, with global $3$-fold rotational symmetry, given as the union of $P_1 \cup r(P_2) \cup P_3 \cup r(P_4) \cup \cdots$, where $r$ represents rotation by $2\pi/6$ and each patch has common centre over the origin. A rigid motion of this is called a \emph{$P_\infty$ tiling}; in \cite{LeeMoo13} they are called {\bf iCW-L} tilings, short for \emph{infinite concurrent w-line tilings}.

\begin{prop} \label{prop: valid => supertile decomposition}
Let $T$ be a valid tiling. Then:
\begin{enumerate}
	\item If the \RI-decorations of $T$ are a rigid motion of $P_\infty$, then $T$ has a unique and locally defined level $n$-supertiling for each $n \in \N$ with prototiles in $\mathcal{P}$, or possibly also one occurrence at each level of the tile in the first class of Figure \ref{fig: omitted}.
	\item If $T$ is an $n$-cycle, then the infinite \RI-edges bound $7$ regions, one of which is a translate of $\varphi^n(t)$, where $t \in \mathcal{P}$ or the second class of Figure \ref{fig: omitted}, and the other $6$ infinite patches have unique supertile decompositions with prototile set $\mathcal{P}$.
	\item If $T$ has an infinite \RI-line without turn $L$, then the infinite patches between infinite \RI-edges appear in certain tilings admitted by $\varphi$ with prototile set $\mathcal{P}$.
	\item Otherwise, $T$ is admitted by $\varphi$ with prototile set $\mathcal{P}$.	
\end{enumerate}
\end{prop}

\begin{proof}
\emph{Case (1):} There are \RI-triangles $\Delta_1$, $\Delta_2$, \ldots with identical centres and $s(\Delta_n) = 2^n$ for all $n \in \N$. By Lemma \ref{lem: triangles => supertiles}, the patch of tiles in each $\Delta_n$ is given by $\varphi^n(t)$, where $t$ is an all inward directed triangle with outer charges given by the charges of the edges of $\Delta_n$. The existence of unique and locally defined level-$n$ supertilings then follows in the same way as the proof of Lemma \ref{lem: primitive and recognisable}. Notice that once the charges for $\Delta_1$ are chosen then the charges for the remaining triangles are determined, since each edge of $\Delta_{n+1}$ is connected to precisely one edge of $\Delta_n$ in the \RI-edge graph (more precisely, we rotate the charges $2\pi / 6$ anti-clockwise between each triangle). In particular, if any $\Delta_k$ is chosen with all edges the same charge, then all other $\Delta_n$ have only this charge too. A tile corresponding to a single-charge level-$1$ triangle does not appear as the substitute of a tile from $\mathcal{P}$, and only occurs once (at the centre, rotated) of a substitute of itself, so it occurs precisely once at each level of supertiling.

\emph{Case (2):} By Lemma \ref{lem: triangles => supertiles}, the finite region bounded between the infinite \RI-edges is given by $\varphi^n(t)$, where $t$ is a triangle with all edges directed outwards. The $6$ infinite regions bounded between infinite \RI-edges are infinite unions of triangular regions of size $2^n$, for $n \in \N$, bounded between \RI-edges and with \RI-decorations determined by Proposition \ref{prop: R1 structure}, which by Lemma \ref{lem: triangles => supertiles} can be expressed as an infinite nested union of $n$-supertiles $\varphi^n(t)$ with a common corner at the intersection of \RI-edges. In fact, the edges of $\varphi^n(t)$ overlapping with the infinite \RI-edges must have the same corresponding charges, and the remaining edge has charge determined by the transfer property, so is a tile in $\mathcal{P}$.

\emph{Case (3):} Take a half-plane on one side of $L$. If there are no other infinite \RI-edges in this region then it can be covered by supertiles $\varphi^n(t)$ for $t \in \mathcal{P}$ which share an outer edge with $L$, by Proposition \ref{prop: R1 structure} and Lemma \ref{lem: triangles => supertiles}. Again, each $t \in \mathcal{P}$ by the charge transfer property. Otherwise, there is an infinite \RI-line with a turn in this half-plane, which is therefore decomposed into three regions with supertile decompositions analogously to the six infinite regions of the previous case. Repeating for the other half-plane, we see that there is a unique covering of $T$ by $n$-supertiles for each $n \in \N$.

\emph{Case (4):} If none of the above cases apply, Proposition \ref{prop: R1 structure} implies that there is an infinite nested sequence $\Delta_1$, $\Delta_2$,\ldots of triangular regions $\Delta_n$ bounded between \RI-edges with sides of length $2^n$ and whose union covers $\R^2$ (if it did not then this would lead to an infinite \RI-edge). By Lemma \ref{lem: triangles => supertiles} these triangles are given by $\varphi^n(t_n)$ for $t_n \in \mathcal{P}'$. However, $\Delta_n$ appears off-centre in $\Delta_{n+1}$ infinitely often, since $T$ is not a $P_\infty$ tiling. Since the second class of tile of Figure \ref{fig: omitted} cannot appear in a substituted tile, and the first class only appears at the centre of a substitute of itself, each $t_n \in \mathcal{P}$, so $T$ is admitted by $\varphi$.
\end{proof}

\begin{remark} \label{rem: bad tiles}
The above shows that valid tilings have unique supertile decompositions of all levels, except for possibly the $P_\infty$ tilings (which are either admitted by $\varphi$ or are admitted by extending to a non-primitive prototile set) and $n$-cycle tilings (which have supertile decompositions of all levels in $6$ infinite regions, covering the whole plane except for a central level $n$-supertile). The above proof shows the \emph{bad tiles}, given as the first two classes of Figure \ref{fig: omitted}, may occur as follows:
\begin{enumerate}
	\item The first class (an \RI-triangle with only one charge of edge) can only be part of a $P_\infty$ tiling or $n$-cycle. For the $P_\infty$ tiling, the tiling is either admitted or there is precisely one such bad tile at each level, which appears at the centre of the next one. This is similar to the situation for the Socolar--Taylor tiling, where the existence of a level-$1$ \RI-triangle forming a vertex of a single colour forces a $P_\infty$ tiling and the remaining decorations. Just as for the Socolar--Taylor tile, there are $8$ possible ways of decorating a $P_\infty$ tiling to get a valid tiling, with $6$ of them admitted by the substitution and the remaining $2$ `defect' tilings having $3$-fold rotational symmetry.
	\item We also see from the above proof that the second class of bad tile of Figure \ref{fig: omitted} may only occur in one role: as a level-$n$ triangle of an $n$-cycle tiling when this tiling has $3$-fold symmetry. In this case, precisely one bad tile occurs at the centre of the next up to level $n$.
\end{enumerate}
\end{remark}

\section{Ergodic and dynamical properties}

We now analyse some fundamental properties of the dynamical system associated to the hull of valid tilings and how they relate to the Socolar--Taylor and Penrose $(1+\epsilon+\epsilon^2)$-tilings.

\subsection{Unique ergodicity}
The supertile decompositions above will allow us to show that all patches in valid tilings appear with well-defined frequencies, from which we may deduce unique ergodicity of the tiling dynamical system. Again, we shall work with the triangular versions of our tilings.

Unique ergodicity of a tiling dynamical system given by the orbit closure of a tiling is equivalent to the tiling having \emph{uniform patch frequencies} \cite{AObook}. The hull of valid tilings is not the orbit closure of a single tiling. Indeed, $\Om$ contains $n$-cycle and fault line tilings (see Theorem \ref{thm: classification}), which contain `defective' finite patches that do not appear in orbit closures of the others. However, we can still prove that the collection of valid tilings has uniform patch frequencies in the following sense:

\begin{definition}
We say that a collection $X$ of tilings has \emph{uniform patch frequencies} if, for each finite patch $P$, there exists some $\delta_P \geq 0$, its \emph{density}, such that for all $\epsilon > 0$ there exists $R > 0$ so that
\[
\left| \frac{\#(P,T,R)}{V_R} - \delta_P \right| < \epsilon
\]
for any $T \in X$, where $\#(P,T,R)$ is the number of occurrences of $P$ inside the ball of radius $R$ centred at the origin of $T$ and $V_R = \pi R^2$ is the area of an $R$-ball.
\end{definition}

We scale our tilings so that the density of all tiles is equal to $1$. Then the densities of all possible patches of the same shape sums to $1$ and we refer to these densities as \emph{patch frequencies}. Some authors demand that frequencies are strictly positive in the definition of uniform patch frequencies, which implies that there is a unique invariant measure and that the open sets have strictly positive measure (that is, the system is \emph{strictly} ergodicity). In our case, there will exist patches with frequency $0$ in the `defect' tilings (i.e., the $n$-cycles, those with fault lines or the $P_\infty$ tilings with $3$-fold symmetry).

The hull of tilings coming from a primitive, recognisable substitution rule always has uniform patch frequencies. Using this in conjunction with the result on supertile decompositions of valid tilings from Proposition \ref{prop: valid => supertile decomposition} allows us to deduce uniform patch frequencies of all valid tilings:

\begin{prop} \label{prop: frequencies}
The collection $\Om$ of valid tilings has uniform patch frequencies.
\end{prop}

\begin{proof}
Take a patch $P$ and valid tiling $T$, which we normalise to have tiles of unit area. Then $P$ appears with frequency $\delta_P$ in any tiling of $\Osub$ (where $\delta_P = 0$ if it does not appear in any tiling of $\Osub$). We claim that $P$ appears with frequency $\delta_P$ in $T$ too.

In case (4) of Proposition \ref{prop: valid => supertile decomposition} the tiling is already in $\Osub$ so this is immediate. Cases (2) and (3) are also clear by the following argument. Take a large $k \in \N$ and remove from $T$ a neighbourhood $N$ of the boundaries of $k$-supertiles, and the central triangle of an $n$-cycle in case (2), so that occurrences of $P$ in $\R^2 \setminus N$ appear only in interiors of $k$-supertiles (which we know appear as patches in $\Osub$). For a large $k$ and $R > 0$, only a negligible proportion of any $R$-ball intersects $N$, and on $\R^2 \setminus N$ occurrences of $P$ occur with the same frequency as in tilings of $\Osub$. So we see that $P$ occurs in $T$ with this same frequency in $T$ too.

Suppose then that $T$ is in case (1) of Proposition \ref{prop: valid => supertile decomposition} and let $\epsilon_1, \epsilon_2 > 0$ be arbitrary. Choose a large integer $k \in \N$. Let $s$ be the smallest number of occurrences of $P$ inside any $k$-supertile $\varphi^k(t)$, for $t \in \mathcal{P}$. Similarly, let $S$ be the largest number of occurrences of $P$ in any level $k$-supertile, where we include potential occurrences of $P$ which may intersect the boundary. By uniform patch frequencies in $\Osub$, we may choose $k$ large enough so that
\[
2^k(\delta_P - \epsilon_1) \leq s \leq S \leq 2^k(\delta_P + \epsilon_1).
\]
Next, take some large $R > 0$. Let $n$ be the smallest number of $k$-supertiles (which may also exclude the single `bad' $k$-supertile which does not appear in $\Osub$, as in Remark \ref{rem: bad tiles}) that are contained in any $R$-ball, and $N$ be the largest number of $k$-supertiles which can intersect any $R$-ball. By making $R$ large enough relative to $2^k$, we can ensure that
\[
\frac{V_R}{2^k}(1-\epsilon_2) \leq n \leq N \leq \frac{V_R}{2^k}(1+\epsilon_2).
\]
By counting occurrences of $P$ in interiors of $k$-supertiles, or ones which may occur along boundaries of supertiles, we see that
\[
n \cdot s \leq \#(P,T,R) \leq N \cdot S + 2^k,
\]
where on the right-hand term we use the (crude) upper bound of $2^k$ possible occurrences of $P$ inside the single `bad' $k$-supertile. Hence
\[
(1-\epsilon_2)(\delta_P - \epsilon_1) \leq \frac{\#(P,T,R)}{V_R} \leq (1+\epsilon_2)(\delta_P + \epsilon_1) + \frac{2^k}{V_R}.
\]
So by making $\epsilon_1$ and $\epsilon_2$ sufficiently small, and $R$ sufficiently large relative to $k$, we see 
\[
\left|\frac{\#(P,T,R)}{V_R} - \delta_P \right|
\]
can be made as small as desired.
\end{proof}

An application of the Ergodic Theorem yields the following:

\begin{cor} \label{cor: uniquely ergodic}
The dynamical system $(\Om,\R^2)$ is uniquely ergodic, and with respect to the unique invariant measure almost every tiling of $\Om$ is in $\Om_\varphi$.
\end{cor}

\begin{proof}
One may consider a transversal to the tiling flow on the continuous hull, which here is a $\Z^2$-shift. This is given, for example, by considering tilings where the origin lies at the centre of a tile. This \emph{canonical transversal} $X$ is a totally disconnected space whose clopen sets have measure determined by patch frequencies, which exist by the above. Unique ergodicity of the continuous hull follows from that of its transversal.

Let $X_r \subset X$ denote the set of valid tilings with $r$-patch at the origin belonging, up to translation, to any substitution tiling (that is, such a patch appears in a $k$-supertile for some $k \in \N$). Then each $X_r$ has measure $1$, as the frequencies of non-admitted patches have frequency $0$ (and so, as above, by the Ergodic Theorem have measure $0$ in $\Om$). Since $X_\varphi = \bigcap_{r > 0} X_r$, continuity from above of the measure implies that $X_\varphi$ has full measure in $X$. The same then follows for the continuous hull.
\end{proof}

\subsection{Topological factors and spectral properties}
Since almost every valid tiling is in the minimal core $\Osub$, for the spectral properties of our tilings it is equivalent to study this more well-behaved space instead, which we shall do for the remainder of the paper. We observe that there is an order two homeomorphism, equivariant with respect to the action of translation, on $\Osub$ (and $\Om$) given by charge-flip. It follows that $(\Osub,\R^2)$ is not an almost everywhere one-to-one extension of its maximal equicontinuous factor (MEF), and so cannot have pure point dynamical spectrum \cite{BarKel13} (or diffraction spectrum \cite{LMS02, Dwo93}). We also consider the quotient space by charge-flip, which we denote by $\Osub^\pm$. Its elements can be naturally viewed as tilings: we remove charge labellings at edges and replace them with labels for each triangle vertex which indicate whether or not the charge flips between the vertex's adjacent edges. It is easily seen that such tilings can be produced by a substitution rule, given by that of Figure \ref{fig: sub} but charge-flip pairs of tiles are identified. This is a substitution of $36$ tiles, $6$ up to rotation equivalence.

The hull of all Socolar--Taylor tilings also has a minimal core generated by substitution, called the \emph{Taylor tilings} in \cite{BGG12}, which we denote by $\Ost$. We recall also from \cite{BGG12} the hulls of \emph{half hex tilings} $\Ohh$, \emph{arrowed half hex tilings} $\Oa$ and \emph{Penrose $(1+\epsilon+\epsilon^2)$-tilings} $\Oe$. Finally, we let $\mathbb{S}_2^2$ be the two-dimensional \emph{maximal equicontinuous factor} (\emph{MEF}) of these dynamical systems, which is given by the two-dimensional dyadic solenoid.

Recall that the arrowed half hex tilings of $\Oa$ are MLD to \emph{standard \RI-tilings}, hex tilings with \RI-decorations as appearing in the Socolar--Taylor tilings \cite{BGG12}. These tilings can be generated by the substitution rule of Figure \ref{fig: sub}, where charge decorations are removed. The half-hex tilings may be considered as such tilings where edges are not offset, but the locations and directions of \RI-turns is known. Since the offset of an \RI-edge is determined on any \RI-triangle with a turn (as this determines the inside of a triangle), the map is $1$-to-$1$ on tilings without a bi-infinite (straight) \RI-edge. On the other hand, it is $2$-to-$1$ on tilings with a bi-infinite \RI-edge, since this may be offset in one of two directions to obtain a tiling of $\Oa$. 

Elements of $\mathbb{S}_2^2$ may be regarded as sequences $T_0$, $T_1$, $T_2$, \ldots of periodic tilings of equilateral triangles of side-length $2^n$, where the standard subdivision of triangles of side-length $2^n$ into four of side-length $2^{n-1}$ takes $T_n$ to $T_{n-1}$. The topology on $\mathbb{S}_2^2$ regards two such sequences as close if their $n$\textsuperscript{th} terms are small translates of each other, for large $n$. The factor map from $\Oa$ to $\mathbb{S}^2_2$ simply records the sequence of unlabelled level-$n$ supertilings (where we use the triangular tile versions of these tilings). Triangles $t$ of $T_n$ which appear at the centres of the next level always correspond to \RI-triangles, with turns directing anti-clockwise about $t$, so the map from $\Oa$ to $\mathbb{S}^2_2$ is $1$-to-$1$ everywhere except over points having an infinite-level vertex (that is, a point which is a vertex of every $T_n$). These points correspond to {\bf CHT} tilings \cite{SocTay11}, where the map is $3$-to-$1$.

By Lemma \ref{lem: charge transfer sufficiency}, there are $2^\ell$ different ways of adding charges to a compatible \RI-tiling to obtain a valid tiling in $\Om$, where $\ell$ is the number of connected components of the \RI-edge graph. We determine this number $\ell$ below, and hence the multiplicity of the fibres of the factor maps to $\Oa$ and the MEF:

\begin{figure}
	\def\svgscale{0.15}
	\centering
	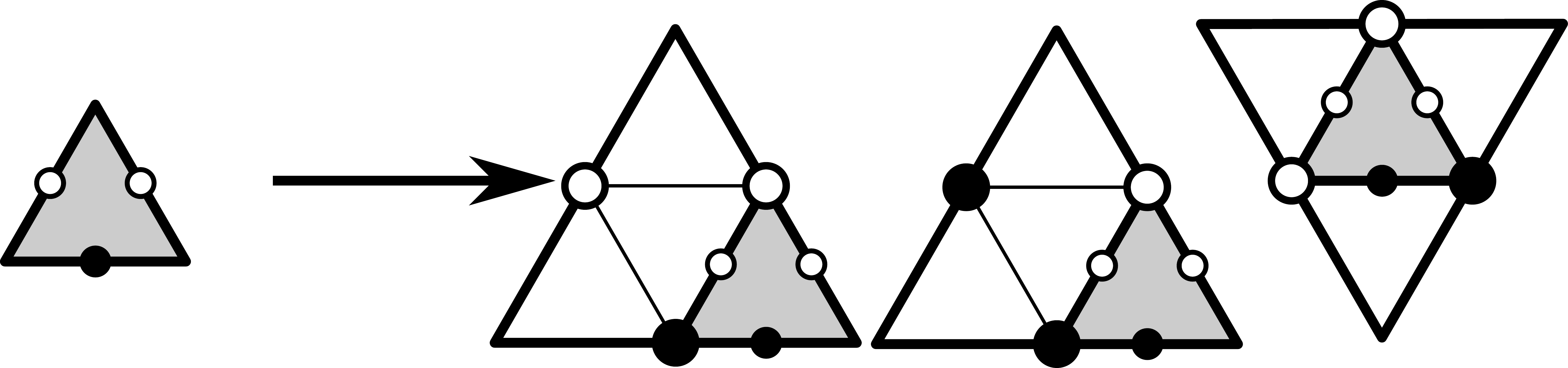
	\caption{The $\alpha$, $\beta$ and $\gamma$-moves, in this respective order above, used to create tilings with $\ell = 2$ components of \RI-edge graph. The white dots on the original tile indicate edges in the same \RI-edge graph component $X$, the black dot indicates the other component $Y$. The $3$ edges of the triangles at the next level are in the component indicated with larger dots.}
	\label{fig: 2-components}
\end{figure}

\begin{lemma} \label{lem: connected components}
Let $T \in \Om$ with $\ell$ connected components of \iRI-edge graph (so $\ell = 1$, $2$ or $3$).
\begin{enumerate}
	\item If $\ell = 3$ then $T$ is a {\bf CHT} tiling (possibly sheared along the bi-infinite \iRI-edge to make a fault-line), a $P_\infty$ tiling or an $n$-cycle. If $T \in \Osub$ then $T$ is a {\bf CHT} or $P_\infty$ tiling. There are $8$ assignments of charges to the {\bf CHT} tiling and $6$ for the $P_\infty$ tiling which give elements of $\Osub$.
	\item If $\ell = 2$ and $T$ has one bi-infinite \iRI-edge then $T \notin \Osub$. If $\ell = 2$ and has no bi-infinite \iRI-edge then $T \in \Osub$ and there is a sequence $\Delta_0 \subset \Delta_1 \subset \Delta_2 \subset \cdots$ covering $\R^2$, where each $\Delta_k$ is a $k$-supertile which, for sufficiently large $k$, is included into $\Delta_{k+1}$ in one of the placements $\alpha$, $\beta$ or $\gamma$ depicted in Figure \ref{fig: 2-components}.
	\item If $\ell = 1$ then either $\iRI(T)$ has an infinite fault line or $T \in \Osub$. Almost every element of $\Om$ has $\ell = 1$.
\end{enumerate}
\end{lemma}

\begin{proof}
It follows from Theorem \ref{thm: classification} that $\ell = 1$, $2$ or $3$. Suppose that $\ell = 3$. If $T$ is a $P_\infty$ tiling or an $n$-cycle then it is easily seen that $\ell = 3$. Otherwise, by Proposition \ref{prop: valid => supertile decomposition}, it has supertile decompositions. If $T$ is covered by a nested union of supertiles, then it must be admitted by $\varphi$. Otherwise, there must be a bi-infinite \RI-edge with two infinite \RI-triangles each side, making it a {\bf CHT} tiling, or a shear of it along the infinite \RI-edge. If $T \in \Osub$ then it follows that $T$ is a $P_\infty$ tiling or a {\bf CHT} tiling. In the former case, there is a nested union $\Delta_0$ $\subset$ $\Delta_1$ $\subset$ $\Delta_2 \subset \cdots$ where each $\Delta_k$ is a $k$-supertile (with inwards directed edges) appearing at the centre of $\Delta_{k+1}$. There are $6$ ways of labelling $\Delta_0$ with charges to obtain a prototile of $\mathcal{P}$, namely those which do not use the same charge on all edges, and this charge determines that of all of the other $\Delta_k$. For the {\bf CHT} tiling, all $8$ valid assignments of charges give a tiling of $\Osub$. This follows from a brief check that all $8$ possible charge assignments to two \RI-triangles meeting opposite across an \RI-edge can be generated by the substitution.

So now suppose that $\ell = 2$. If $T$ has a bi-infinite \RI-edge $E$ then the \RI-edge graph on one half-plane side of $E$ must be in the same component as $E$, whereas the other side must have an infinite \RI-triangle. This can happen in $\Om$ but not in $\Osub$. Suppose then that $T$ does not have a bi-infinite \RI-edge. Since $T$ is not an $n$-cycle (which has $\ell = 3$), there must be the stated sequence $\Delta_0 \subset \Delta_1 \subset \Delta_2 \subset \cdots$ of $k$-supertiles by Proposition \ref{prop: valid => supertile decomposition}; indeed, if any such sequence covered only part of the plane then there would be an infinite \RI-line. All of the supertiles are substitutes of tiles from $\mathcal{P}$, or else the tiling would be a $P_\infty$ tiling. It follows that $T$ is admitted by $\varphi$. Since $\ell = 2$, for sufficiently large $k$, two of the edges of $\Delta_k$ belong to the same component $X$ of the \RI-edge graph, and the other edge belongs to the other component $Y$. This limits the placement of $\Delta_k$ into $\Delta_{k+1}$ to one of $\alpha$, $\beta$ or $\gamma$ in Figure \ref{fig: 2-components}.

Finally, suppose that $\ell = 1$. By Proposition \ref{prop: valid => supertile decomposition}, either $T$ is admitted by $\varphi$ or has an infinite \RI-line, splitting into half-planes that are patches of certain tilings from $\Osub$. If $\RI(T)$ has a fault line then $T \notin \Osub$. If there is no fault-line in the \RI-decorations then, since $\ell = 1$, the charges of one half-plane determine those on the other and $T$ itself must be admitted by $\varphi$. Finally, we claim that almost every element of $\Osub$ (and hence $\Om$, by Corollary \ref{cor: uniquely ergodic}) has $\ell = 1$. Indeed, $k$-supertiles of all levels are positioned periodically, with two given sides being connected in the \RI-edge graph at the next level whenever they appear in one of four possible locations in the $(k+1)$-supertile containing them. So generically, any two given edges are eventually connected in some sufficiently large supertile.
\end{proof}

\begin{figure}
\begin{center}
\begin{tikzpicture}
\node at (1,-5.5) {$\scriptstyle (0,0,i)$};
\node[vertex] (vertexe) at (1,-5.5)   {$\,\quad$}
	edge [->,>=latex,out=70,in=110,loop,thick] node[above,pos=0.5]{$\scriptstyle \alpha$} (vertexe);
\node at (1,-6.2) {$\scriptstyle (0,i,i)$};
\node[vertex] (vertexd) at (1,-6.2)   {$\,\quad$}
	edge [->,>=latex,out=250,in=290,loop,thick] node[below,pos=0.5]{$\scriptstyle \alpha$} (vertexd);
\node at (1,-3.1) {$\scriptstyle (0,0,0)$};
\node[vertex] (vertexa) at (1,-3)   {$\,\quad$}
	edge [->,>=latex,out=70,in=110,loop,thick] node[above,pos=0.5]{$\scriptstyle \alpha$} (vertexa)
	edge [->,>=latex,out=60,in=120,loop,thick,looseness=10] node[above,pos=0.5]{$\scriptstyle \beta$} (vertexa);
\node at (-1.1,-4) {$\scriptstyle (i,i,i)$};
\node[vertex] (vertexb) at (-1,-4)   {$\quad$}
	edge [->,>=latex,out=110,in=150,loop,thick] node[above,pos=0.5]{$\scriptstyle \gamma$} (vertexb)
	edge [->,>=latex,out=300,in=160,thick] node[right,pos=0.4]{$\scriptstyle \gamma$} (vertexe)
	edge [->,>=latex,out=285,in=170,thick] node[right,pos=0.5]{$\scriptstyle \gamma$} (vertexd)
	edge [<-,>=latex,out=270,in=185,thick] node[left,pos=0.5]{$\scriptstyle \beta$} (vertexd)
	edge [<-,>=latex,out=254,in=200,thick,looseness=1.2] node[left,pos=0.5]{$\scriptstyle \alpha$} (vertexd) 
	edge [->,>=latex,out=50,in=190,thick] node[above,pos=0.5]{$\scriptstyle \gamma$} (vertexa);
\node at (3.1,-4) {$\scriptstyle (0,i,0)$};
\node[vertex] (vertexc) at (3,-4)   {$\quad$}
	edge [->,>=latex,out=70,in=30,loop,thick] node[above,pos=0.5]{$\scriptstyle \alpha$} (vertexb)
	edge [->,>=latex,out=240,in=20,thick] node[left,pos=0.4]{$\scriptstyle \beta$} (vertexe)
	edge [<-,>=latex,out=130,in=350,thick] node[above,pos=0.5]{$\scriptstyle \beta$} (vertexa)
	edge [->,>=latex,out=260,in=10,thick] node[right,pos=0.5]{$\scriptstyle \beta$} (vertexd)
	edge [<-,>=latex,out=180,in=0,thick] node[above,pos=0.5]{$\scriptstyle \gamma$} (vertexb);
\end{tikzpicture}
\end{center}
\caption{Graph of nodes the different triangle types, connected by all possible $\alpha$, $\beta$ and $\gamma$-moves, see Figure \ref{fig: 2-components}. An infinite path in this graph, which does not end in an infinite tail of $\alpha$ or infinite tail of $\gamma$, defines a tiling with $\ell = 2$ components of \RI-edge graph.}
\label{fig: SFT}
\end{figure}

We can give a precise description of which tilings have $\ell = 2$ connected components of \RI-edge graph. A triangle (positioned with horizontal bottom edge) may have edges directed inwards or outwards, to which we associate a tuple $(l,r,b) \in \{i,o\}^3$, with the first, second and third coordinate corresponding to the left, right and bottom edge, respectively, $o$ denoting `outwards' and $i$ denoting `inwards'. Because of the placement of triangles, we may only apply an $\alpha$ or $\beta$ move to an $(o,x,y)$ triangle (where $x$, $y \in \{0,i\}$ may be chosen arbitrarily), and a $\gamma$ move may only be applied to an $(i,i,i)$ triangle (in particular, we cannot have an $(i,x,y)$ triangle unless $x=y=i$). Applying an $\alpha$ move to an $(o,x,y)$ triangle results in a $(z,x,y)$ triangle, whereas $\beta$ results in a $(y,z,x)$ triangle, both for $z$ arbitrary (where, for the latter, note that we rotate the resulting configuration making $X$ the bottom edge, since in the new triangle the edges $Y$ are in the same component and take on the roles of the $X$-labelled edges from the step before). A $\gamma$-move can lead to any triangle type.

Thus, any path around the graph of Figure \ref{fig: SFT} builds an \RI-tiling with $X$ and $Y$ in different path-components. Moreover, so long as the path does not end in repeated applications of $\alpha$ (which would mean that $T$ is a {\bf CHT} tiling, if $T \in \Osub$) or of $\gamma$ (which would result in a $P_\infty$ tiling), then it is easily seen that the sequence of supertiles covers the plane. Moreover, in this case, since two edges of a triangle are reconnected at either an $\alpha$ or $\beta$ move, we see that the resulting tiling has precisely $2$ connected components. Conversely, a valid tiling with two components (without bi-infinite \RI-line), is covered by supertiles whose placements eventually follow an infinite path in the graph of Figure \ref{fig: SFT} and does not have an infinite tail of only $\alpha$ or only $\gamma$.

Having determined the multiplicities of fibres to $\Oa$ (and the MEF), we now compare the situation to the Penrose $(1+\epsilon+\epsilon^2)$ and Socolar--Taylor tilings which, despite not being MLD (the \v{C}ech cohomologies of $\Ost$ and $\Om_\epsilon$ differ \cite{BGG12}), somewhat amazingly map to the MEF with fibres of identical multiplicity \cite{BGG12,LeeMoo13}. We summarise with the table in Figure \ref{fig: table}

\begin{figure}
\begin{tabular}{ |p{3cm}||p{3cm}|p{3cm}|p{3cm}|p{3cm}  }
 \hline
  & \multicolumn{2}{|c|}{Size of fibre of $\ast \to \Oa$} & \\
 \hline
   Type of $\mathbb{S}_2^2$ point & $\Ost$ and $\Oe$ & $\Osub^\pm$ & Size of fibre of $\Oa \to \mathbb{S}^2_2$ \\
 \hline\hline
 {\bf CHT} (w and a-sing) & $2$ & $4$         & $6$ \\
 \hline
 $P_\infty$ ({\bf iCW-L})     & $6$ & $3$         & $1$ \\
 \hline
 a-gen, w-gen             & $1$ & $1$ or $2$  & $1$ \\
 \hline
 a-sing, w-gen            & $1$ & $1$         & $2$ \\
 \hline
 a-gen, w-sing            & $2$ & $1$ or $2$  & $1$ \\
 \hline
\end{tabular}
\caption{Comparing fibre multiplicities with the Penrose $(1+\epsilon+\epsilon^2)$ and Socolar--Taylor tilings.}
\label{fig: table}
\end{figure}

In table in Figure \ref{fig: table}, `gen' is shorthand for `generic' and `sing' for `singular'. The a-lines of a point in the solenoid (or associated tiling) are defined in \cite{LeeMoo13}; they correspond to the locations of the (non-offset) \RI-edges, which are arranged periodically, and more sparsely at increasing levels of the hierarchy. Tilings where every $a$-line has a finite level are called a-generic, otherwise they are a-singular (that is, a tiling is a-singular if it has an infinite \RI-line). The w-lines are defined somewhat similarly. They form another triangular grid at each level, each line cutting through lines of reflective symmetry of the triangles of a-lines, and are related to the carry of information of flags (or colours) in the \RII-rule of the Socolar--Taylor tilings. We refer the reader to \cite{LeeMoo13} for more details on these notions. The first column of values of the above table are given in \cite{LeeMoo13} and \cite{BGG12}.

For the second column of values, the first two rows are the content of Lemma \ref{lem: connected components}(1). For the third row it is not hard to see (analogously to the proof that almost all tilings have $\ell = 1$ components of \RI-edge graph in the proof of Lemma \ref{lem: connected components}(3)) that almost all tilings have no infinite $a$- or $w$-lines (these are also called \emph{generic} in \cite{LeeMoo13}). So almost every tiling is generic and has $\ell = 1$, although we may also construct $a$- and $w$- generic tilings with $\ell = 2$, which will happen for most infinite paths in the graph of Figure \ref{fig: SFT} (the path $\beta^\infty$ gives just one example). The fourth row (a-sing, w-gen) corresponds to tilings with an infinite \RI-line which are not {\bf CHT}, which have $\ell = 1$ by Lemma \ref{lem: connected components}. Finally, we may find tilings without infinite \RI-edges but an infinite $w$-line with either $\ell = 1$ or $\ell = 2$. Following the path $\alpha \gamma^6 \alpha \gamma^6 \alpha \gamma^6 \cdots$ of the graph in Figure \ref{fig: SFT} constructs an example with $\ell = 2$. Examples with $\ell = 1$ are simple to construct, by successively placing supertiles with common lines of reflection and not following a path in the graph of Figure \ref{fig: SFT}.

Since $\Osub^\pm$ factors almost everywhere one-to-one to its MEF, it has pure point dynamical and diffraction spectrum (see, for example \cite{LMS02, Dwo93}), and a regular model set structure by \cite[Theorem 6]{BLM07}. This factor map has different multiplicity of fibres to $\Ost$ and $\Oe$ and so cannot be topologically conjugate to either of them. We summarise these main results together with results from \cite{BGG12,LeeMoo13}:

\begin{cor}
We have the following diagram of maps of $\R^2$-dynamical systems:
\[
\begin{tikzcd}
                             & \Oe \ar[dr] &            &             & \\ 
\Om                          & \Ost \ar[r] & \Oa \ar[r] & \Ohh \ar[r] & \mathbb{S}^2_2 \\
\Osub \ar[u,hook] \ar[r,"f"] & \Osub^\pm \ar[ur]        &             & 
\end{tikzcd}
\]
Each map is a factor map, except for the inclusion $\Osub \hookrightarrow \Om$, which has image of full measure. All factor maps are almost everywhere one-to-one, except for the map $f$ which is a $2$-to-$1$ covering map. In particular, $(\Om,\R^2)$ and $(\Osub,\R^2)$ do not have pure point dynamical spectrum, but $(\Osub^\pm,\R^2)$ does. Moreover, none of these systems are topologically conjugate. The tilings spaces $\Oe$, $\Ost$, $\Oa$ and $\Osub^\pm$ each have the structure of regular model sets.
\end{cor}

\bibliographystyle{siam}
\bibliography{biblio}

\end{document}

%% file: charge_transfer.pdf_tex
\begingroup%
  \makeatletter%
  \providecommand\color[2][]{%
    \errmessage{(Inkscape) Color is used for the text in Inkscape, but the package 'color.sty' is not loaded}%
    \renewcommand\color[2][]{}%
  }%
  \providecommand\transparent[1]{%
    \errmessage{(Inkscape) Transparency is used (non-zero) for the text in Inkscape, but the package 'transparent.sty' is not loaded}%
    \renewcommand\transparent[1]{}%
  }%
  \providecommand\rotatebox[2]{#2}%
  \newcommand*\fsize{\dimexpr\f@size pt\relax}%
  \newcommand*\lineheight[1]{\fontsize{\fsize}{#1\fsize}\selectfont}%
  \ifx\svgwidth\undefined%
    \setlength{\unitlength}{2679.04794768bp}%
    \ifx\svgscale\undefined%
      \relax%
    \else%
      \setlength{\unitlength}{\unitlength * \real{\svgscale}}%
    \fi%
  \else%
    \setlength{\unitlength}{\svgwidth}%
  \fi%
  \global\let\svgwidth\undefined%
  \global\let\svgscale\undefined%
  \makeatother%
  \begin{picture}(1,0.53626967)%
    \lineheight{1}%
    \setlength\tabcolsep{0pt}%
    \put(0,0){\includegraphics[width=\unitlength,page=1]{charge_transfer.pdf}}%
    \put(0.17132029,0.35799748){\color[rgb]{0,0,0}\makebox(0,0)[lt]{\begin{minipage}{0.51710796\unitlength}\raggedright \end{minipage}}}%
    \put(0.44604882,0.31249524){\color[rgb]{0,0,0}\makebox(0,0)[lt]{\lineheight{1.25}\smash{\begin{tabular}[t]{l}$E_1$\end{tabular}}}}%
    \put(0.58538634,0.08734633){\color[rgb]{0,0,0}\makebox(0,0)[lt]{\lineheight{1.25}\smash{\begin{tabular}[t]{l}$E_2$\end{tabular}}}}%
    \put(0,0){\includegraphics[width=\unitlength,page=2]{charge_transfer.pdf}}%
    \put(0.19332037,0.29339128){\color[rgb]{0,0,0}\makebox(0,0)[lt]{\begin{minipage}{0.05151614\unitlength}\raggedright $c^*$\\ \end{minipage}}}%
    \put(0.11070342,0.27467623){\color[rgb]{0,0,0}\makebox(0,0)[lt]{\lineheight{1.25}\smash{\begin{tabular}[t]{l}$c$\end{tabular}}}}%
    \put(-0.00123936,0.31061486){\color[rgb]{0,0,0}\makebox(0,0)[lt]{\lineheight{1.25}\smash{\begin{tabular}[t]{l}$E_1$\end{tabular}}}}%
    \put(0.16164779,0.08605313){\color[rgb]{0,0,0}\makebox(0,0)[lt]{\lineheight{1.25}\smash{\begin{tabular}[t]{l}$E_2$\end{tabular}}}}%
    \put(0,0){\includegraphics[width=\unitlength,page=3]{charge_transfer.pdf}}%
    \put(0.22674254,0.2743356){\color[rgb]{0,0,0}\makebox(0,0)[lt]{\lineheight{1.25}\smash{\begin{tabular}[t]{l}$c$\end{tabular}}}}%
    \put(0.24913848,0.23258308){\color[rgb]{0,0,0}\makebox(0,0)[lt]{\lineheight{1.25}\smash{\begin{tabular}[t]{l}$c$\end{tabular}}}}%
    \put(0.3091278,0.33416501){\color[rgb]{0,0,0}\makebox(0,0)[lt]{\lineheight{1.25}\smash{\begin{tabular}[t]{l}$c$\end{tabular}}}}%
    \put(0.34881104,0.42544131){\color[rgb]{0,0,0}\makebox(0,0)[lt]{\begin{minipage}{0.05151614\unitlength}\raggedright $c^*$\\ \end{minipage}}}%
    \put(0.22835243,0.22275736){\color[rgb]{0,0,0}\makebox(0,0)[lt]{\begin{minipage}{0.05151614\unitlength}\raggedright $c^*$\\ \end{minipage}}}%
    \put(0.2914612,0.32313946){\color[rgb]{0,0,0}\makebox(0,0)[lt]{\begin{minipage}{0.05151614\unitlength}\raggedright $c^*$\\ \end{minipage}}}%
    \put(0.19194867,0.13340074){\color[rgb]{0,0,0}\makebox(0,0)[lt]{\lineheight{1.25}\smash{\begin{tabular}[t]{l}$c$\end{tabular}}}}%
    \put(0.15509242,0.26923206){\color[rgb]{0,0,0}\makebox(0,0)[lt]{\lineheight{1.25}\smash{\begin{tabular}[t]{l}$t_1$\end{tabular}}}}%
    \put(0.27155187,0.27155165){\color[rgb]{0,0,0}\makebox(0,0)[lt]{\lineheight{1.25}\smash{\begin{tabular}[t]{l}$t_2$\end{tabular}}}}%
    \put(0.63775632,0.29595082){\color[rgb]{0,0,0}\makebox(0,0)[lt]{\begin{minipage}{0.05151614\unitlength}\raggedright $c^*$\\ \end{minipage}}}%
    \put(0.55513941,0.27723576){\color[rgb]{0,0,0}\makebox(0,0)[lt]{\lineheight{1.25}\smash{\begin{tabular}[t]{l}$c$\end{tabular}}}}%
    \put(0.67117848,0.27689513){\color[rgb]{0,0,0}\makebox(0,0)[lt]{\lineheight{1.25}\smash{\begin{tabular}[t]{l}$c$\end{tabular}}}}%
    \put(0.75356376,0.33672451){\color[rgb]{0,0,0}\makebox(0,0)[lt]{\lineheight{1.25}\smash{\begin{tabular}[t]{l}$c$\end{tabular}}}}%
    \put(0.79324699,0.42800084){\color[rgb]{0,0,0}\makebox(0,0)[lt]{\begin{minipage}{0.05151614\unitlength}\raggedright $c^*$\\ \end{minipage}}}%
    \put(0.67278838,0.22531689){\color[rgb]{0,0,0}\makebox(0,0)[lt]{\begin{minipage}{0.05151614\unitlength}\raggedright $c^*$\\ \end{minipage}}}%
    \put(0.73589715,0.325699){\color[rgb]{0,0,0}\makebox(0,0)[lt]{\begin{minipage}{0.05151614\unitlength}\raggedright $c^*$\\ \end{minipage}}}%
    \put(0.63638462,0.13596029){\color[rgb]{0,0,0}\makebox(0,0)[lt]{\lineheight{1.25}\smash{\begin{tabular}[t]{l}$c$\end{tabular}}}}%
    \put(0.5995284,0.2717916){\color[rgb]{0,0,0}\makebox(0,0)[lt]{\lineheight{1.25}\smash{\begin{tabular}[t]{l}$t_1$\end{tabular}}}}%
    \put(0.70926889,0.275231){\color[rgb]{0,0,0}\makebox(0,0)[lt]{\lineheight{1.25}\smash{\begin{tabular}[t]{l}$t_2$\end{tabular}}}}%
    \put(0.69665757,0.23752965){\color[rgb]{0,0,0}\makebox(0,0)[lt]{\lineheight{1.25}\smash{\begin{tabular}[t]{l}$c$\end{tabular}}}}%
    \put(0.75745828,0.29314484){\color[rgb]{0,0,0}\makebox(0,0)[lt]{\begin{minipage}{0.05151614\unitlength}\raggedright $c^*$\\ \end{minipage}}}%
    \put(0.31449705,0.29178508){\color[rgb]{0,0,0}\makebox(0,0)[lt]{\begin{minipage}{0.05151614\unitlength}\raggedright $c^*$\\ \end{minipage}}}%
  \end{picture}%
\endgroup%

%% file: spiral.pdf_tex
\begingroup%
  \makeatletter%
  \providecommand\color[2][]{%
    \errmessage{(Inkscape) Color is used for the text in Inkscape, but the package 'color.sty' is not loaded}%
    \renewcommand\color[2][]{}%
  }%
  \providecommand\transparent[1]{%
    \errmessage{(Inkscape) Transparency is used (non-zero) for the text in Inkscape, but the package 'transparent.sty' is not loaded}%
    \renewcommand\transparent[1]{}%
  }%
  \providecommand\rotatebox[2]{#2}%
  \ifx\svgwidth\undefined%
    \setlength{\unitlength}{7309.23351378bp}%
    \ifx\svgscale\undefined%
      \relax%
    \else%
      \setlength{\unitlength}{\unitlength * \real{\svgscale}}%
    \fi%
  \else%
    \setlength{\unitlength}{\svgwidth}%
  \fi%
  \global\let\svgwidth\undefined%
  \global\let\svgscale\undefined%
  \makeatother%
  \begin{picture}(1,0.89026159)%
    \put(0,0){\includegraphics[width=\unitlength,page=1]{spiral.pdf}}%
    \put(-0.03854451,-0.90695363){\color[rgb]{0,0,0}\makebox(0,0)[lt]{\begin{minipage}{0.18953518\unitlength}\raggedright \end{minipage}}}%
    \put(0.03483937,0.78974224){\color[rgb]{0,0,0}\makebox(0,0)[lt]{\begin{minipage}{0.06860207\unitlength}\raggedright $E_1$\end{minipage}}}%
    \put(0.97603511,0.84165357){\color[rgb]{0,0,0}\makebox(0,0)[lt]{\begin{minipage}{0.07681085\unitlength}\raggedright $E_2$\end{minipage}}}%
    \put(0.50711183,0.03719767){\color[rgb]{0,0,0}\makebox(0,0)[lt]{\begin{minipage}{0.05583749\unitlength}\raggedright $E_3$\end{minipage}}}%
    \put(0.3600799,0.5474996){\color[rgb]{0,0,0}\makebox(0,0)[lt]{\begin{minipage}{0.16385591\unitlength}\raggedright $E_4=E_7$\end{minipage}}}%
    \put(0,0){\includegraphics[width=\unitlength,page=2]{spiral.pdf}}%
    \put(0.51712643,0.86218481){\color[rgb]{0,0,0}\makebox(0,0)[lt]{\begin{minipage}{0.09996301\unitlength}\raggedright $\Delta_1$\end{minipage}}}%
    \put(0.87042825,0.37120702){\color[rgb]{0,0,0}\makebox(0,0)[lt]{\begin{minipage}{0.09996301\unitlength}\raggedright $\Delta_2$\end{minipage}}}%
    \put(0.28575822,0.22492582){\color[rgb]{0,0,0}\makebox(0,0)[lt]{\begin{minipage}{0.09996301\unitlength}\raggedright $\Delta_3$\end{minipage}}}%
    \put(0.40578687,0.60864436){\color[rgb]{0,0,0}\makebox(0,0)[lt]{\begin{minipage}{0.09996301\unitlength}\raggedright $\Delta_4$\end{minipage}}}%
    \put(0.60114576,0.44028086){\color[rgb]{0,0,0}\makebox(0,0)[lt]{\begin{minipage}{0.09996301\unitlength}\raggedright $\Delta_5$\end{minipage}}}%
    \put(0.45282391,0.4048486){\color[rgb]{0,0,0}\makebox(0,0)[lt]{\begin{minipage}{0.09996301\unitlength}\raggedright $\Delta_6$\end{minipage}}}%
    \put(0,0){\includegraphics[width=\unitlength,page=3]{spiral.pdf}}%
  \end{picture}%
\endgroup%

%% file: row_of_triangles.pdf_tex
\begingroup%
  \makeatletter%
  \providecommand\color[2][]{%
    \errmessage{(Inkscape) Color is used for the text in Inkscape, but the package 'color.sty' is not loaded}%
    \renewcommand\color[2][]{}%
  }%
  \providecommand\transparent[1]{%
    \errmessage{(Inkscape) Transparency is used (non-zero) for the text in Inkscape, but the package 'transparent.sty' is not loaded}%
    \renewcommand\transparent[1]{}%
  }%
  \providecommand\rotatebox[2]{#2}%
  \newcommand*\fsize{\dimexpr\f@size pt\relax}%
  \newcommand*\lineheight[1]{\fontsize{\fsize}{#1\fsize}\selectfont}%
  \ifx\svgwidth\undefined%
    \setlength{\unitlength}{4276.80764026bp}%
    \ifx\svgscale\undefined%
      \relax%
    \else%
      \setlength{\unitlength}{\unitlength * \real{\svgscale}}%
    \fi%
  \else%
    \setlength{\unitlength}{\svgwidth}%
  \fi%
  \global\let\svgwidth\undefined%
  \global\let\svgscale\undefined%
  \makeatother%
  \begin{picture}(1,0.59111345)%
    \lineheight{1}%
    \setlength\tabcolsep{0pt}%
    \put(0,0){\includegraphics[width=\unitlength,page=1]{row_of_triangles.pdf}}%
    \put(0.13058614,0.33830062){\color[rgb]{0,0,0}\makebox(0,0)[lt]{\begin{minipage}{0.09353247\unitlength}\raggedright $\Delta = \Delta_i$\end{minipage}}}%
    \put(0.28980244,0.14676024){\color[rgb]{0,0,0}\makebox(0,0)[lt]{\begin{minipage}{0.09353247\unitlength}\raggedright $\Delta_{i+1}$\end{minipage}}}%
    \put(0.43468154,0.22211238){\color[rgb]{0,0,0}\makebox(0,0)[lt]{\begin{minipage}{0.07465516\unitlength}\raggedright $\Delta_{i+2}$\end{minipage}}}%
    \put(0.45845087,0.53426884){\color[rgb]{0,0,0}\makebox(0,0)[lt]{\lineheight{1.25}\smash{\begin{tabular}[t]{l}$t$\end{tabular}}}}%
    \put(0.32156376,0.41836425){\color[rgb]{0,0,0}\makebox(0,0)[lt]{\begin{minipage}{0.07723515\unitlength}\raggedright $E$\end{minipage}}}%
    \put(0.07794374,0.04393196){\color[rgb]{0,0,0}\makebox(0,0)[lt]{\begin{minipage}{0.08786384\unitlength}\raggedright $L$\end{minipage}}}%
    \put(0.58174368,0.1279806){\color[rgb]{0,0,0}\makebox(0,0)[lt]{\lineheight{1.25}\smash{\begin{tabular}[t]{l}$\Delta_{i+3}$\end{tabular}}}}%
    \put(0.69137808,0.33983071){\color[rgb]{0,0,0}\makebox(0,0)[lt]{\lineheight{1.25}\smash{\begin{tabular}[t]{l}$\Delta' = \Delta_{i+4}$\end{tabular}}}}%
    \put(0.58295691,0.42000293){\color[rgb]{0,0,0}\makebox(0,0)[lt]{\begin{minipage}{0.0517263\unitlength}\raggedright $E'$\end{minipage}}}%
    \put(0,0){\includegraphics[width=\unitlength,page=2]{row_of_triangles.pdf}}%
  \end{picture}%
\endgroup%

%% file: standard_patch.pdf_tex
\begingroup%
  \makeatletter%
  \providecommand\color[2][]{%
    \errmessage{(Inkscape) Color is used for the text in Inkscape, but the package 'color.sty' is not loaded}%
    \renewcommand\color[2][]{}%
  }%
  \providecommand\transparent[1]{%
    \errmessage{(Inkscape) Transparency is used (non-zero) for the text in Inkscape, but the package 'transparent.sty' is not loaded}%
    \renewcommand\transparent[1]{}%
  }%
  \providecommand\rotatebox[2]{#2}%
  \newcommand*\fsize{\dimexpr\f@size pt\relax}%
  \newcommand*\lineheight[1]{\fontsize{\fsize}{#1\fsize}\selectfont}%
  \ifx\svgwidth\undefined%
    \setlength{\unitlength}{11232.29799766bp}%
    \ifx\svgscale\undefined%
      \relax%
    \else%
      \setlength{\unitlength}{\unitlength * \real{\svgscale}}%
    \fi%
  \else%
    \setlength{\unitlength}{\svgwidth}%
  \fi%
  \global\let\svgwidth\undefined%
  \global\let\svgscale\undefined%
  \makeatother%
  \begin{picture}(1,0.22517241)%
    \lineheight{1}%
    \setlength\tabcolsep{0pt}%
    \put(0,0){\includegraphics[width=\unitlength,page=1]{standard_patch.pdf}}%
    \put(0.9257733,0.19487504){\color[rgb]{0,0,0}\makebox(0,0)[lt]{\begin{minipage}{0.05654646\unitlength}\raggedright $P_n$\end{minipage}}}%
  \end{picture}%
\endgroup%

%% file: consistency.pdf_tex
\begingroup%
  \makeatletter%
  \providecommand\color[2][]{%
    \errmessage{(Inkscape) Color is used for the text in Inkscape, but the package 'color.sty' is not loaded}%
    \renewcommand\color[2][]{}%
  }%
  \providecommand\transparent[1]{%
    \errmessage{(Inkscape) Transparency is used (non-zero) for the text in Inkscape, but the package 'transparent.sty' is not loaded}%
    \renewcommand\transparent[1]{}%
  }%
  \providecommand\rotatebox[2]{#2}%
  \newcommand*\fsize{\dimexpr\f@size pt\relax}%
  \newcommand*\lineheight[1]{\fontsize{\fsize}{#1\fsize}\selectfont}%
  \ifx\svgwidth\undefined%
    \setlength{\unitlength}{1990.34306564bp}%
    \ifx\svgscale\undefined%
      \relax%
    \else%
      \setlength{\unitlength}{\unitlength * \real{\svgscale}}%
    \fi%
  \else%
    \setlength{\unitlength}{\svgwidth}%
  \fi%
  \global\let\svgwidth\undefined%
  \global\let\svgscale\undefined%
  \makeatother%
  \begin{picture}(1,0.18598395)%
    \lineheight{1}%
    \setlength\tabcolsep{0pt}%
    \put(0,0){\includegraphics[width=\unitlength,page=1]{consistency.pdf}}%
    \put(0.01045699,0.10042681){\color[rgb]{0,0,0}\makebox(0,0)[lt]{\begin{minipage}{0.05558512\unitlength}\raggedright $c$\end{minipage}}}%
    \put(0,0){\includegraphics[width=\unitlength,page=2]{consistency.pdf}}%
    \put(0.71185409,0.15511668){\color[rgb]{0,0,0}\makebox(0,0)[lt]{\begin{minipage}{0.0450429\unitlength}\raggedright $x$\end{minipage}}}%
    \put(0.76809444,0.15261847){\color[rgb]{0,0,0}\makebox(0,0)[lt]{\begin{minipage}{0.0450429\unitlength}\raggedright $x$\end{minipage}}}%
    \put(0.71260094,0.04734582){\color[rgb]{0,0,0}\makebox(0,0)[lt]{\begin{minipage}{0.08009769\unitlength}\raggedright $x^*$\end{minipage}}}%
    \put(0.77343273,0.04759158){\color[rgb]{0,0,0}\makebox(0,0)[lt]{\begin{minipage}{0.08923326\unitlength}\raggedright $x^*$\end{minipage}}}%
    \put(0.92503151,0.13799189){\color[rgb]{0,0,0}\makebox(0,0)[lt]{\lineheight{1.25}\smash{\begin{tabular}[t]{l}$y$\end{tabular}}}}%
    \put(0.96249672,0.08537257){\color[rgb]{0,0,0}\makebox(0,0)[lt]{\lineheight{1.25}\smash{\begin{tabular}[t]{l}$y$\end{tabular}}}}%
    \put(0.83897511,0.1105635){\color[rgb]{0,0,0}\makebox(0,0)[lt]{\begin{minipage}{0.0768903\unitlength}\raggedright $y^*$\end{minipage}}}%
    \put(0.86904143,0.05671567){\color[rgb]{0,0,0}\makebox(0,0)[lt]{\begin{minipage}{0.07993549\unitlength}\raggedright $y^*$\end{minipage}}}%
    \put(0.04040921,0.15487112){\color[rgb]{0,0,0}\makebox(0,0)[lt]{\begin{minipage}{0.05558512\unitlength}\raggedright $c$\end{minipage}}}%
    \put(0.16769765,0.1020177){\color[rgb]{0,0,0}\makebox(0,0)[lt]{\begin{minipage}{0.05558512\unitlength}\raggedright $c$\end{minipage}}}%
    \put(0.19525932,0.15563374){\color[rgb]{0,0,0}\makebox(0,0)[lt]{\begin{minipage}{0.05558512\unitlength}\raggedright $c$\end{minipage}}}%
    \put(0.11130523,0.05238935){\color[rgb]{0,0,0}\makebox(0,0)[lt]{\begin{minipage}{0.07051908\unitlength}\raggedright $c^*$\end{minipage}}}%
    \put(0.12833016,0.10686232){\color[rgb]{0,0,0}\makebox(0,0)[lt]{\begin{minipage}{0.07051908\unitlength}\raggedright $c^*$\end{minipage}}}%
    \put(0.26570783,0.05141601){\color[rgb]{0,0,0}\makebox(0,0)[lt]{\begin{minipage}{0.07051908\unitlength}\raggedright $c^*$\end{minipage}}}%
    \put(0.28594836,0.10912324){\color[rgb]{0,0,0}\makebox(0,0)[lt]{\begin{minipage}{0.07051908\unitlength}\raggedright $c^*$\end{minipage}}}%
    \put(0,0){\includegraphics[width=\unitlength,page=3]{consistency.pdf}}%
    \put(0.6211802,0.09715328){\color[rgb]{0,0,0}\makebox(0,0)[lt]{\begin{minipage}{0.07051908\unitlength}\raggedright $y^*$\end{minipage}}}%
    \put(0.4997281,0.09991311){\color[rgb]{0,0,0}\makebox(0,0)[lt]{\begin{minipage}{0.05558512\unitlength}\raggedright $y$\end{minipage}}}%
    \put(0.53391735,0.15696401){\color[rgb]{0,0,0}\makebox(0,0)[lt]{\begin{minipage}{0.05558512\unitlength}\raggedright $y$\end{minipage}}}%
    \put(0.59657964,0.05342066){\color[rgb]{0,0,0}\makebox(0,0)[lt]{\begin{minipage}{0.07051908\unitlength}\raggedright $y^*$\end{minipage}}}%
    \put(0,0){\includegraphics[width=\unitlength,page=4]{consistency.pdf}}%
    \put(0.37345418,0.13784881){\color[rgb]{0,0,0}\makebox(0,0)[lt]{\lineheight{1.25}\smash{\begin{tabular}[t]{l}$x$\end{tabular}}}}%
    \put(0.43021762,0.14099189){\color[rgb]{0,0,0}\makebox(0,0)[lt]{\lineheight{1.25}\smash{\begin{tabular}[t]{l}$x$\end{tabular}}}}%
    \put(0.37840204,0.05016039){\color[rgb]{0,0,0}\makebox(0,0)[lt]{\begin{minipage}{0.09059352\unitlength}\raggedright $x^*$\end{minipage}}}%
    \put(0.43490574,0.05241369){\color[rgb]{0,0,0}\makebox(0,0)[lt]{\begin{minipage}{0.07536776\unitlength}\raggedright $x^*$\end{minipage}}}%
  \end{picture}%
\endgroup%

%% file: edge_graph.pdf_tex
\begingroup%
  \makeatletter%
  \providecommand\color[2][]{%
    \errmessage{(Inkscape) Color is used for the text in Inkscape, but the package 'color.sty' is not loaded}%
    \renewcommand\color[2][]{}%
  }%
  \providecommand\transparent[1]{%
    \errmessage{(Inkscape) Transparency is used (non-zero) for the text in Inkscape, but the package 'transparent.sty' is not loaded}%
    \renewcommand\transparent[1]{}%
  }%
  \providecommand\rotatebox[2]{#2}%
  \newcommand*\fsize{\dimexpr\f@size pt\relax}%
  \newcommand*\lineheight[1]{\fontsize{\fsize}{#1\fsize}\selectfont}%
  \ifx\svgwidth\undefined%
    \setlength{\unitlength}{5102.36220472bp}%
    \ifx\svgscale\undefined%
      \relax%
    \else%
      \setlength{\unitlength}{\unitlength * \real{\svgscale}}%
    \fi%
  \else%
    \setlength{\unitlength}{\svgwidth}%
  \fi%
  \global\let\svgwidth\undefined%
  \global\let\svgscale\undefined%
  \makeatother%
  \begin{picture}(1,0.5)%
    \lineheight{1}%
    \setlength\tabcolsep{0pt}%
    \put(0,0){\includegraphics[width=\unitlength,page=1]{edge_graph.pdf}}%
  \end{picture}%
\endgroup%

%% file: triangle_sequence.pdf_tex
\begingroup%
  \makeatletter%
  \providecommand\color[2][]{%
    \errmessage{(Inkscape) Color is used for the text in Inkscape, but the package 'color.sty' is not loaded}%
    \renewcommand\color[2][]{}%
  }%
  \providecommand\transparent[1]{%
    \errmessage{(Inkscape) Transparency is used (non-zero) for the text in Inkscape, but the package 'transparent.sty' is not loaded}%
    \renewcommand\transparent[1]{}%
  }%
  \providecommand\rotatebox[2]{#2}%
  \ifx\svgwidth\undefined%
    \setlength{\unitlength}{5484.71306748bp}%
    \ifx\svgscale\undefined%
      \relax%
    \else%
      \setlength{\unitlength}{\unitlength * \real{\svgscale}}%
    \fi%
  \else%
    \setlength{\unitlength}{\svgwidth}%
  \fi%
  \global\let\svgwidth\undefined%
  \global\let\svgscale\undefined%
  \makeatother%
  \begin{picture}(1,0.50662913)%
    \put(0,0){\includegraphics[width=\unitlength,page=1]{triangle_sequence.pdf}}%
    \put(0.02393895,0.02846678){\color[rgb]{0,0,0}\makebox(0,0)[lb]{\smash{$t$}}}%
    \put(0.08078227,0.02846678){\color[rgb]{0,0,0}\makebox(0,0)[lb]{\smash{$t_1$}}}%
    \put(0.13244426,0.02887186){\color[rgb]{0,0,0}\makebox(0,0)[lb]{\smash{$t_2$}}}%
    \put(0.18928756,0.02887186){\color[rgb]{0,0,0}\makebox(0,0)[lb]{\smash{$t_3$}}}%
    \put(0.24613083,0.02887187){\color[rgb]{0,0,0}\makebox(0,0)[lb]{\smash{$t_4$}}}%
    \put(0.30297419,0.02887187){\color[rgb]{0,0,0}\makebox(0,0)[lb]{\smash{$t_5$}}}%
    \put(0.3598175,0.02887187){\color[rgb]{0,0,0}\makebox(0,0)[lb]{\smash{$t_6$}}}%
    \put(0.41666071,0.02887184){\color[rgb]{0,0,0}\makebox(0,0)[lb]{\smash{$t_7$}}}%
    \put(0.47350404,0.02887184){\color[rgb]{0,0,0}\makebox(0,0)[lb]{\smash{$t_8$}}}%
    \put(0.5303473,0.02887186){\color[rgb]{0,0,0}\makebox(0,0)[lb]{\smash{$t_9$}}}%
    \put(0.58719059,0.02887186){\color[rgb]{0,0,0}\makebox(0,0)[lb]{\smash{$t_{10}$}}}%
    \put(0.64403392,0.02887188){\color[rgb]{0,0,0}\makebox(0,0)[lb]{\smash{$t_{11}$}}}%
    \put(0.70087724,0.02887188){\color[rgb]{0,0,0}\makebox(0,0)[lb]{\smash{$t_{12}$}}}%
    \put(0.75772053,0.02887188){\color[rgb]{0,0,0}\makebox(0,0)[lb]{\smash{$t_{13}$}}}%
    \put(0.81456383,0.02887189){\color[rgb]{0,0,0}\makebox(0,0)[lb]{\smash{$t_{14}$}}}%
    \put(0.87140706,0.02887185){\color[rgb]{0,0,0}\makebox(0,0)[lb]{\smash{$t_{15}$}}}%
    \put(0.92825029,0.02887188){\color[rgb]{0,0,0}\makebox(0,0)[lb]{\smash{$t_{16}$}}}%
    \put(0.43341242,0.30088831){\color[rgb]{0,0,0}\makebox(0,0)[lt]{\begin{minipage}{0.10168522\unitlength}\raggedright \end{minipage}}}%
    \put(0.47298797,0.246433){\color[rgb]{0,0,0}\makebox(0,0)[lb]{\smash{$\Delta_8$}}}%
    \put(0.236341,0.14881621){\color[rgb]{0,0,0}\makebox(0,0)[lb]{\smash{$\Delta_4$}}}%
    \put(0.12645741,0.09164499){\color[rgb]{0,0,0}\makebox(0,0)[lb]{\smash{$\Delta_2$}}}%
    \put(0.35600852,0.09045145){\color[rgb]{0,0,0}\makebox(0,0)[lb]{\smash{$\Delta_6$}}}%
    \put(0.69857234,0.15243186){\color[rgb]{0,0,0}\makebox(0,0)[lb]{\smash{$\Delta_{12}$}}}%
    \put(0.7031371,0.41995487){\color[rgb]{0,0,0}\makebox(0,0)[lb]{\smash{$s$}}}%
    \put(0.30529015,0.47620553){\color[rgb]{0,0,0}\makebox(0,0)[lt]{\begin{minipage}{0.04361287\unitlength}\raggedright \end{minipage}}}%
    \put(0.30034917,0.46261953){\color[rgb]{0,0,0}\makebox(0,0)[lb]{\smash{$E_2$}}}%
    \put(0.97272602,0.01146197){\color[rgb]{0,0,0}\makebox(0,0)[lb]{\smash{$E_1$}}}%
  \end{picture}%
\endgroup%

%% file: cycle.pdf_tex
\begingroup%
  \makeatletter%
  \providecommand\color[2][]{%
    \errmessage{(Inkscape) Color is used for the text in Inkscape, but the package 'color.sty' is not loaded}%
    \renewcommand\color[2][]{}%
  }%
  \providecommand\transparent[1]{%
    \errmessage{(Inkscape) Transparency is used (non-zero) for the text in Inkscape, but the package 'transparent.sty' is not loaded}%
    \renewcommand\transparent[1]{}%
  }%
  \providecommand\rotatebox[2]{#2}%
  \ifx\svgwidth\undefined%
    \setlength{\unitlength}{5695.0180318bp}%
    \ifx\svgscale\undefined%
      \relax%
    \else%
      \setlength{\unitlength}{\unitlength * \real{\svgscale}}%
    \fi%
  \else%
    \setlength{\unitlength}{\svgwidth}%
  \fi%
  \global\let\svgwidth\undefined%
  \global\let\svgscale\undefined%
  \makeatother%
  \begin{picture}(1,0.4680567)%
    \put(0.46263394,0.28040929){\color[rgb]{0,0,0}\makebox(0,0)[lt]{\begin{minipage}{0.0979302\unitlength}\raggedright \end{minipage}}}%
    \put(0.33924296,0.44925241){\color[rgb]{0,0,0}\makebox(0,0)[lt]{\begin{minipage}{0.04200234\unitlength}\raggedright \end{minipage}}}%
    \put(0,0){\includegraphics[width=\unitlength,page=1]{cycle.pdf}}%
    \put(0.39174589,0.17765611){\color[rgb]{0,0,0}\makebox(0,0)[lt]{\begin{minipage}{0.07435549\unitlength}\raggedright $E_1$\end{minipage}}}%
    \put(0.03312254,0.15103127){\color[rgb]{0,0,0}\makebox(0,0)[lt]{\begin{minipage}{0.07435551\unitlength}\raggedright $E_2$\end{minipage}}}%
    \put(0.22856109,0.45577976){\color[rgb]{0,0,0}\makebox(0,0)[lt]{\begin{minipage}{0.03206935\unitlength}\raggedright \end{minipage}}}%
    \put(0.22922824,0.45015624){\color[rgb]{0,0,0}\makebox(0,0)[lt]{\begin{minipage}{0.07435549\unitlength}\raggedright $E_3$\end{minipage}}}%
    \put(0.18353649,0.04545796){\color[rgb]{0,0,0}\makebox(0,0)[lt]{\begin{minipage}{0.07435549\unitlength}\raggedright $E_1'$\end{minipage}}}%
    \put(0.05667818,0.41581655){\color[rgb]{0,0,0}\makebox(0,0)[lt]{\begin{minipage}{0.0743555\unitlength}\raggedright $E_2'$\end{minipage}}}%
    \put(0.41880625,0.4481695){\color[rgb]{0,0,0}\makebox(0,0)[lt]{\begin{minipage}{0.07435551\unitlength}\raggedright $E_3'$\end{minipage}}}%
    \put(0.7167958,0.05511556){\color[rgb]{0,0,0}\makebox(0,0)[lt]{\begin{minipage}{0.07435551\unitlength}\raggedright $E_1$\end{minipage}}}%
    \put(0.93208423,0.1538779){\color[rgb]{0,0,0}\makebox(0,0)[lt]{\begin{minipage}{0.07435551\unitlength}\raggedright $E_2'$\end{minipage}}}%
    \put(0.9222667,0.41497335){\color[rgb]{0,0,0}\makebox(0,0)[lt]{\begin{minipage}{0.07435551\unitlength}\raggedright $E_2$\end{minipage}}}%
    \put(0.73966777,0.46038129){\color[rgb]{0,0,0}\makebox(0,0)[lt]{\begin{minipage}{0.07435551\unitlength}\raggedright $E_3'$\end{minipage}}}%
    \put(0.55746847,0.44392091){\color[rgb]{0,0,0}\makebox(0,0)[lt]{\begin{minipage}{0.07435551\unitlength}\raggedright $E_3'$\end{minipage}}}%
    \put(0.57449654,0.17771701){\color[rgb]{0,0,0}\makebox(0,0)[lt]{\begin{minipage}{0.07435551\unitlength}\raggedright $E_1'$\end{minipage}}}%
  \end{picture}%
\endgroup%

%% file: MLD.pdf_tex
\begingroup%
  \makeatletter%
  \providecommand\color[2][]{%
    \errmessage{(Inkscape) Color is used for the text in Inkscape, but the package 'color.sty' is not loaded}%
    \renewcommand\color[2][]{}%
  }%
  \providecommand\transparent[1]{%
    \errmessage{(Inkscape) Transparency is used (non-zero) for the text in Inkscape, but the package 'transparent.sty' is not loaded}%
    \renewcommand\transparent[1]{}%
  }%
  \providecommand\rotatebox[2]{#2}%
  \ifx\svgwidth\undefined%
    \setlength{\unitlength}{1567.34953633bp}%
    \ifx\svgscale\undefined%
      \relax%
    \else%
      \setlength{\unitlength}{\unitlength * \real{\svgscale}}%
    \fi%
  \else%
    \setlength{\unitlength}{\svgwidth}%
  \fi%
  \global\let\svgwidth\undefined%
  \global\let\svgscale\undefined%
  \makeatother%
  \begin{picture}(1,0.92408392)%
    \put(0,0){\includegraphics[width=\unitlength,page=1]{MLD.pdf}}%
  \end{picture}%
\endgroup%

%% file: sub.pdf_tex
\begingroup%
  \makeatletter%
  \providecommand\color[2][]{%
    \errmessage{(Inkscape) Color is used for the text in Inkscape, but the package 'color.sty' is not loaded}%
    \renewcommand\color[2][]{}%
  }%
  \providecommand\transparent[1]{%
    \errmessage{(Inkscape) Transparency is used (non-zero) for the text in Inkscape, but the package 'transparent.sty' is not loaded}%
    \renewcommand\transparent[1]{}%
  }%
  \providecommand\rotatebox[2]{#2}%
  \ifx\svgwidth\undefined%
    \setlength{\unitlength}{2834.04150775bp}%
    \ifx\svgscale\undefined%
      \relax%
    \else%
      \setlength{\unitlength}{\unitlength * \real{\svgscale}}%
    \fi%
  \else%
    \setlength{\unitlength}{\svgwidth}%
  \fi%
  \global\let\svgwidth\undefined%
  \global\let\svgscale\undefined%
  \makeatother%
  \begin{picture}(1,0.63387802)%
    \put(0,0){\includegraphics[width=\unitlength,page=1]{sub.pdf}}%
  \end{picture}%
\endgroup%

%% file: omitted.pdf_tex
\begingroup%
  \makeatletter%
  \providecommand\color[2][]{%
    \errmessage{(Inkscape) Color is used for the text in Inkscape, but the package 'color.sty' is not loaded}%
    \renewcommand\color[2][]{}%
  }%
  \providecommand\transparent[1]{%
    \errmessage{(Inkscape) Transparency is used (non-zero) for the text in Inkscape, but the package 'transparent.sty' is not loaded}%
    \renewcommand\transparent[1]{}%
  }%
  \providecommand\rotatebox[2]{#2}%
  \ifx\svgwidth\undefined%
    \setlength{\unitlength}{10825.99693959bp}%
    \ifx\svgscale\undefined%
      \relax%
    \else%
      \setlength{\unitlength}{\unitlength * \real{\svgscale}}%
    \fi%
  \else%
    \setlength{\unitlength}{\svgwidth}%
  \fi%
  \global\let\svgwidth\undefined%
  \global\let\svgscale\undefined%
  \makeatother%
  \begin{picture}(1,0.29988695)%
    \put(0,0){\includegraphics[width=\unitlength,page=1]{omitted.pdf}}%
  \end{picture}%
\endgroup%

%% file: 2-components.pdf_tex
\begingroup%
  \makeatletter%
  \providecommand\color[2][]{%
    \errmessage{(Inkscape) Color is used for the text in Inkscape, but the package 'color.sty' is not loaded}%
    \renewcommand\color[2][]{}%
  }%
  \providecommand\transparent[1]{%
    \errmessage{(Inkscape) Transparency is used (non-zero) for the text in Inkscape, but the package 'transparent.sty' is not loaded}%
    \renewcommand\transparent[1]{}%
  }%
  \providecommand\rotatebox[2]{#2}%
  \ifx\svgwidth\undefined%
    \setlength{\unitlength}{2700.34446897bp}%
    \ifx\svgscale\undefined%
      \relax%
    \else%
      \setlength{\unitlength}{\unitlength * \real{\svgscale}}%
    \fi%
  \else%
    \setlength{\unitlength}{\svgwidth}%
  \fi%
  \global\let\svgwidth\undefined%
  \global\let\svgscale\undefined%
  \makeatother%
  \begin{picture}(1,0.23445362)%
    \put(0,0){\includegraphics[width=\unitlength,page=1]{2-components.pdf}}%
  \end{picture}%
\endgroup%